\RequirePackage{fix-cm}
\documentclass[smallextended]{svjour3}       % onecolumn (second format)
\smartqed  % flush right qed marks, e.g. at end of proof
\usepackage{graphicx}
\usepackage{amsmath}
\usepackage{amsbsy}
\usepackage{amssymb}
\usepackage{textcomp}
\usepackage{graphicx}
\usepackage{amsfonts}
\usepackage{longtable}
\usepackage{multirow}
\usepackage{lscape}
\usepackage{url}
\usepackage{citeref}
\usepackage{subfig}
\usepackage{longtable}
\usepackage{lscape}
\usepackage{hhline}
\usepackage[utf8]{inputenc}
\usepackage{colortbl,array} % für farbige cells
\usepackage{multirow,bigdelim}
\usepackage{booktabs}
\usepackage[linesnumbered,ruled]{algorithm2e}
\usepackage{color, colortbl}
\usepackage{epstopdf}

\def\D{\displaystyle}

\newtheorem{thm}{Theorem}%[section]
\newtheorem{cor}[thm]{Corollary}
\newtheorem{lem}[thm]{Lemma}

\newtheorem{exam}[thm]{Example}
\date{}

\begin{document}
\title{Two globally convergent nonmonotone trust-region methods for unconstrained optimization}
%\title{Exploiting an efficient nonmonotone strategy in trust-region methods for unconstrained optimization}
\author{Masoud Ahookhosh \and Susan Ghaderi}

\institute{M. Ahookhosh
           \at Faculty of Mathematics, University of Vienna, Oskar-Morgenstern-Platz 1, 1090 Vienna,
            Austria\\
           \email{masoud.ahookhosh@univie.ac.at}
           \and S. Ghaderi
           \at Department of Mathematics, Faculty of Science, Razi University,
              Kermanshah, Iran \\ 
           %Department of Mathematics, Azarbaijan Shahid Madani University, Tabriz, Iran\\
           \email{susan.ghaderi23@gmail.com}                
            }
% \date{Received: date / Accepted: date}
\maketitle

%###########################################################################################################
\begin{abstract}
This paper addresses some trust-region methods equipped with nonmonotone strategies for solving nonlinear unconstrained optimization problems. More specifically, the importance of using nonmonotone techniques in nonlinear optimization is motivated, then two new nonmonotone terms are proposed, and their combinations into the traditional trust-region framework are studied. The global convergence to first- and second-order stationary points and local superlinear and quadratic convergence rates for both algorithms are established. Numerical experiments on the \textsf{CUTEst} test collection of unconstrained problems and some highly nonlinear test functions are reported, where a comparison among state-of-the-art nonmonotone trust-region methods show the efficiency of the proposed nonmonotne schemes. 

\end{abstract}

% ################################################
% ################################################
\section{Introduction} \label{intro}
In this paper we consider the unconstrained minimization problem   
\begin{equation}\label{e.func} 
\begin{array}{ll} 
\textrm{minimize}   &  ~  f(x)\\ 
\textrm{subject to} &  ~  x\in \mathbb{R}^n,
\end{array} 
\end{equation} 
where $f:\mathbb{R}^n\rightarrow \mathbb{R}$  is a real-valued nonlinear function, which is bounded and continuously-differentiable. We suppose that first- or second-order black-box oracle of $f$ is available.\\
{\bf Motivation \& history.} Trust-region methods, also called restricted step methods \cite{Fle}, are a class of iterative schemes developed to solve convex or nonconvex  optimization problems, see, for example, \cite{ConGT}. They also developed for nonsmooth problems, see \cite{ConGT,DenLT,Yua,GraYY}. Trust-region methods have strong convergence properties, are reliable and robust in computation, and can handle ill-conditioned problems, cf. \cite{RojS,SchSB}. Let $x_k$ be the current iteration. In trust-region framework the objective $f$ is approximated by a simple model in a specific region around $x_k$ such that it is an acceptable approximation of the original objective, which is called region of trust. Afterward, the model is minimized subject to the trust-region constraint to find a new trial point $d_k$. Hence the simple model means that it can be minimized much easier than the original objective function.  If the founded model is an adequate approximation of the objective function within the trust-region, then the point $x_{k+1}=x_k+d_k$ is accepted by the trust-region method and the region can be expanded for the next iteration; conversely, if the approximation is poor,  the region is contracted and the model is minimized within the contracted region. This scheme will be continued until finding an acceptable trial step $d_k$ guaranteeing an acceptable agreement between the model and the objective function.

Several quadratic and non-quadratic models have been proposed to approximate the objective function in optimization, see \cite{Dav, FerKS, Sor1, Sun}, however, the conic and quadratic models are more popular, see \cite{DiS,Sor2}. If the approximated model is quadratic, i.e., 
\begin{equation} \label{e.qk}
q_k(d) := f_k + g_k^T d + \frac{1}{2} d^T B_k d,
\end{equation}
where $f_k = f(x_k)$, $g_k = \nabla f(x_k)$, and $B_k \approx \nabla^2 f(x_k)$, the trust-region method can be considered as a globally convergent generalization on classical Newton's method. Then the trust-region sunproblem is defined by
\begin{equation} \label{e.sub}
\begin{array}{ll}
\mathrm{minimize}   & ~ q_k(d),\\
\mathrm{subject~to} & ~ \|d\| \leq \delta_k.
\end{array}
\end{equation}
Hence the trust-region is commonly a norm ball $C$ defined by
\[
C := \{d\in \mathbb{R}^n \mid \|d\| \leq \delta_k\},
\]
where $\delta_k > 0$ is a real number called trust-region radius, and $\|\cdot\|$ is any norm in $\mathbb{R}^n$, cf. \cite{HalT}. Since $C$ is compact and the model is continuous, the trust-region subproblem attains its minimizer on the set $C$. The most computational cost of trust-region methods relates to minimizing the model over the trust-region $C$. Hence finding efficient schemes for solving (\ref{e.sub}) has received much attention during past few decades, see \cite{ErwG,ErwGG,GolLRT,MorS,Ste}. Once the step $d$ is computed, the quality of the model in the trust-region is evaluated by a ratio of the actual reduction of objective, $f_k-f(x_k+d)$, to the predicted reduction of model, $q_k(0)-q_k(d)$, i.e.,
\begin{equation} \label{e.rk}
r_k = \frac{f_k-f(x_k+d)}{q_k(0)-q_k(d)}.
\end{equation}
For a prescribed positive constant $\mu_1 \in (0, 1]$, if $r_k \geq \mu_1$, the model provides a reasonable approximation, the step is accepted, i.e., $x_{k+1}=x_k+d_k$, and the trust-region $C$ can be expanded for the next step. Otherwise, the trust-region $C$ should be contracted by decreasing the radius $\delta_k$ and the subproblem (\ref{e.sub}) is solved in the reduced region. This scheme is continued until that the step $d$ accepted by trust-region test $r_k \geq \mu_1$. Our discussion can be summarized in the following algorithm: 

\vspace{-3mm}
%%%%%%%%%%%%%%%%%%%%%%%%%%%%%%%%%%
\begin{algorithm}[h] \label{a.ttr}
\DontPrintSemicolon % Some LaTeX compilers require you to use \dontprintsemicolon instead
\KwIn{$x_0 \in \mathbb{R}^n$, $B_0 \in \mathbb{R}^{n \times n}$, $k_{max}$; $0<\mu_1 \leq \mu_2 \leq 1$, $0 < \rho_1 \leq 1 \leq \rho_2$,
$\varepsilon >0$;}
\KwOut{$x_b$; $f_b$;}
\Begin{
    $\delta_0\leftarrow \|g_0\|$;~ $k \leftarrow 0$;\;
    \While {$\|g_k\| \geq \varepsilon~~ \&~~ k \leq k_{max}$}{
        solve the subproblem (\ref{e.sub}) to specify $d_k$;\; 
        $\widehat{x}_k \leftarrow x_k + d_k$;~ compute $f(\widehat{x}_k)$;\;
        determine $r_k$ using (\ref{e.rk});\;
        \While {$r_k < \mu_1$}{ 
            $\delta_k \leftarrow \rho_1 \delta_k $;\;
            solve the subproblem (\ref{e.sub}) to specify $d_k$;\;
            $\widehat{x}_k \leftarrow x_k + d_k$; ~ compute $f(\widehat{x}_k)$\;
            determine $r_k$ using (\ref{e.rk});\;
        }
        $x_{k+1} \leftarrow \widehat{x}_k$;\;
        \If{$r_k \geq \mu_2$}{
        $\delta_{k + 1} \leftarrow \rho_2 \delta_k $;
        }
        update $B_{k+1}$;~ $k \leftarrow k+1$;\;
    }
    $x_b \leftarrow x_k$;~ $f_b \leftarrow f_k$;
}
\caption{ {\bf TTR} (traditional trust-region algorithm)}
\end{algorithm}
%%%%%%%%%%%%%%%

\vspace{-4mm}
In Algorithm 1, it follows from  $r_k \geq \mu_1$ and $q_k(0) - q_k(d_k) > 0$ that
\[
f_k - f_{k+1} \geq \mu_1 (q_k(0) - q_k(d_k)) > 0,
\]
implying $f_{k+1} \leq f_k$. This means that the sequence of function values $\{f_k\}$ is monotonically decreasing, i.e., the traditional trust-region method is also called the {\bf \emph{monotone trust-region}} method. This feature seems natural for minimization schemes, however, it slows down the convergence of TTR to a minimizer if the objective involves a curved narrow valley, see \cite{AhoA1,GriLL1}. To observe the effect of nonmonotonicity on TTR, we study the next example.

%%%%%%%%%%%%
\begin{exam} 
Consider the two-dimensional Nesterov-Chebysheve-Rosenbrock function , cf. \cite{GurO},
\begin{equation*}
f(x_1,x_2) =\frac{1}{4}( x _1 - 1 )^ 2 +( x _2 - 2x_1^2 + 1 )^2 ,
\end{equation*}
where we solve the problem (\ref{e.func}) by Newton's method and TTR with the initial point $x_0 = (-0.61, -1)$. It is clear that $(1, 1)$ is the optimizer. The implementation indicates that Newton's method needs 7 iterations and 8 function evaluations, while monotone trust-region  method needs 22 iterations and 24 function evaluations.
We depict the contour plot of the objective and iterations as well as a diagram for function values versus iteration attained by these two algorithms in Figure \ref{fcon1}. Subfigure (a) of Figure \ref{fcon1} shows that the iterations of TTR follow the bottom of the valley in contrast to those for Newton's method that can go up and down to reach the $\varepsilon$-solution with the accuracy parameter $\varepsilon = 10^{-5}$. We see that Newton's method attains larger step compared with those of TTR. Subfigure (b) of Figure \ref{fcon1} illustrates function values versus iterations for both algorithms showing that the related function values of TTR decreases monotonically, while it is fluctuated nonmonotonically for Newton's method.
\begin{figure}[h] 
\centering 
\subfloat[][Nes-Cheb-Rosen contour plot \& iterations]{\includegraphics[width=7.7cm]{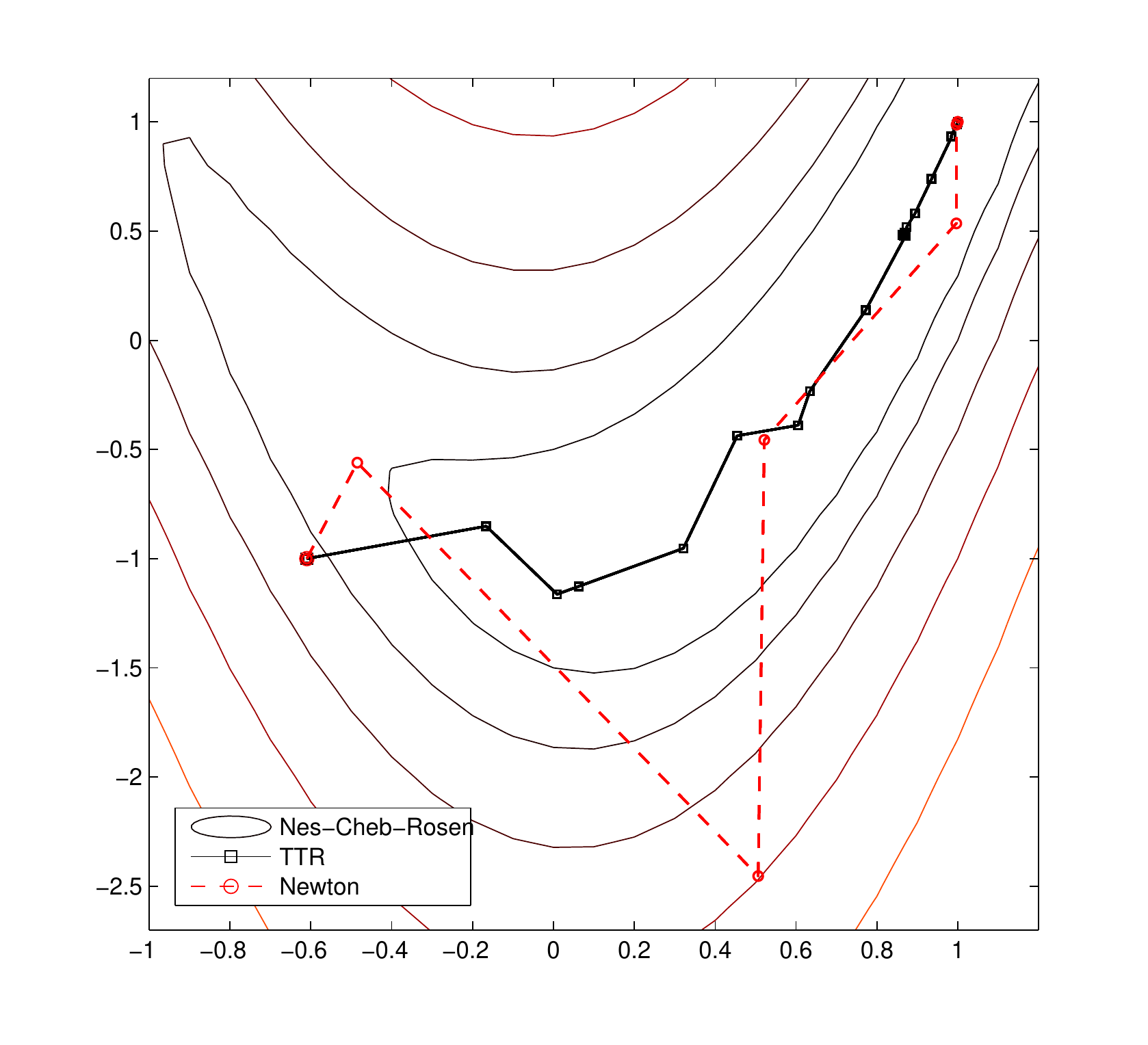}}% 
\qquad 
\subfloat[][function values versus iterations]{\includegraphics[width=7.7cm]{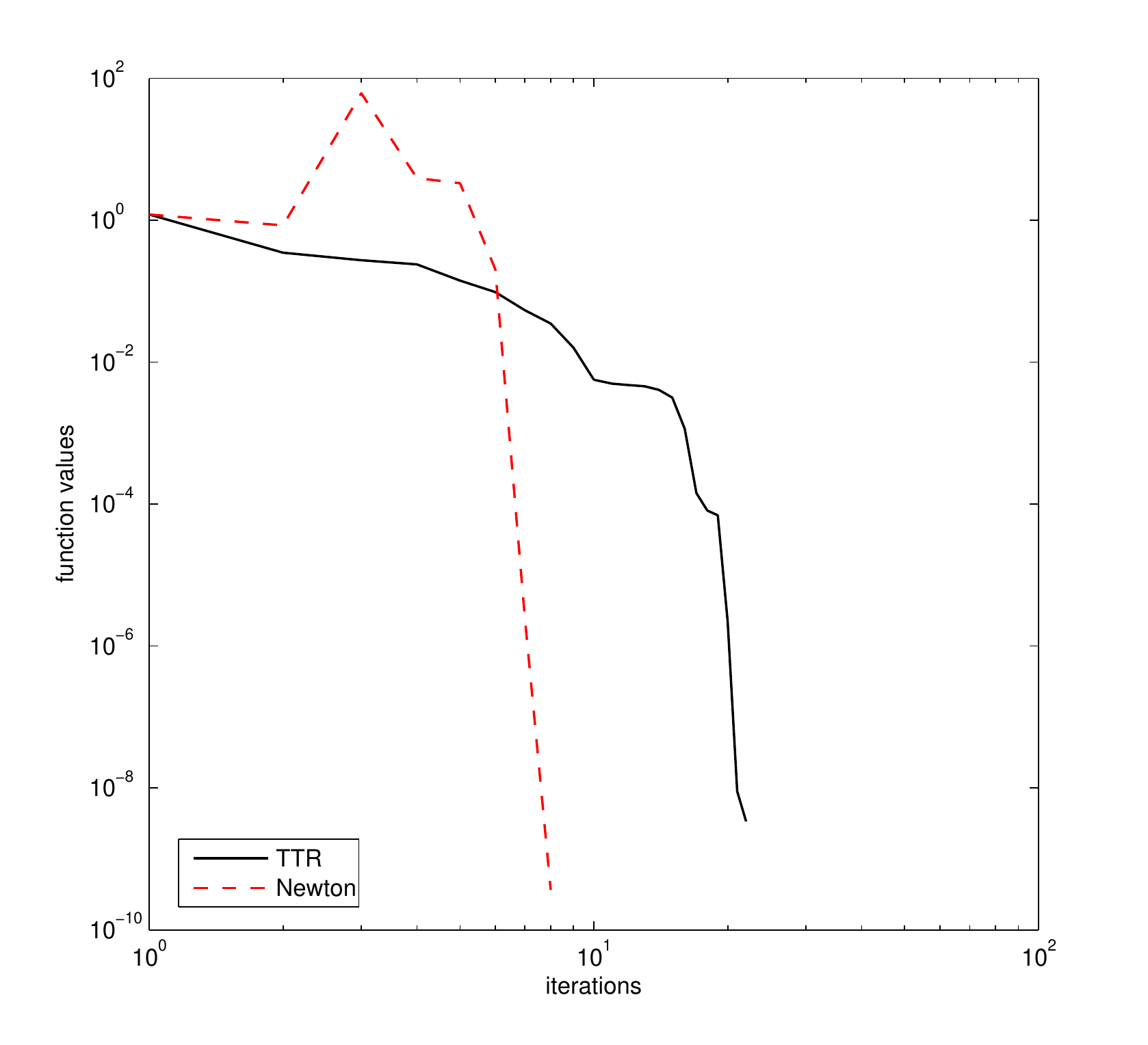}} 
\caption{A comparison between Newton's method and TTR: Subfigure (a) illustrates the contour plot of the two-dimensional Nesterov-Chebysheve-Rosenbrock function and iterations of Newton and TTR; Subfigure (b) shows the diagram of function values versus iterations.} 
\label{fcon1} 
\end{figure} 
\end{exam}
%%%%%%%%%% 
 
In general the monotonicity may result to the slow iterative schemes for highly nonlinear or badly-scaled problems. To avoiding this algorithmic limitation, the idea of nonmonotone strategies has been proposed traced back to the watch-dog technique to overcome the Martos effect for constrained optimization \cite{ChaPLP}. To improve the performance of Armijo's line search, {\sc Grippo} et al. in 1986 \cite{GriLL1} proposed the modified Armijo's rule 
\[
f(x_k + \alpha_k d_k) \leq f_{l(k)} +  \sigma \alpha_k g_k^T d_k,~~ k= 0,1,2,\cdots,
\]
with the step-size $\alpha_{k} >0$, $\sigma \in (0, 1/2)$, and
\begin{equation}\label{e.grip}
f_{l(k)} = \max_{0 \leq j \leq m(k)} \{f_{k-j}\},
\end{equation}
where $m(0)=0$, $m(k)\leq \min \{m(k-1)+1, N\} $ for nonnegetive integer $N$. It was shown that the associated scheme is globally convergent, and numerical results reported in {\sc Grippo} et al. \cite{GriLL1} and {\sc Toint} \cite{Toi1} showed the effectiveness of the proposed idea. Motivated by these results, the nonmonotone strategies has received much attention during past few decades. For example, in 2004, {\sc Zhang \& Hager} in \cite{ZhaH} proposed the nonmonotone term  
\begin{equation}\label{Hag}
\begin{array}{l}
  C_{k}=\left\{
  \begin{array}{ll}
    f_0 &~~ \mathrm{if}~ k=0, \\
    (\eta_{k-1}Q_{k-1}C_{k-1}+f(x_k))/Q_k&~~ \mathrm{if}~ k \geq 1, 
  \end{array}  \right. 
~~~
  Q_{k}=\left\{%
  \begin{array}{ll}
    1 & ~~\mathrm{if}~ k=0, \\
    \eta_{k-1}Q_{k-1}+1 &~~ \mathrm{if}~ k \geq 1, 
  \end{array} \right. 
\end{array}
\end{equation}
where $0 \leq \eta_{min} \leq \eta_{k-1} \leq \eta_{max} \leq 1$. Recently, {\sc Mo} et al. in \cite{MoLY} and {\sc Ahookhosh} et al. in \cite{AhoAB} studied the nonmonotone term
\begin{equation}\label{e.tk1} 
 D_k  = \left\{ 
\begin{array}{ll} 
 f_k     &~~ \mathrm{if}~ k = 1,\\ 
 \eta_{k}D_{k-1} +(1-\eta_{k})f_k &~~ \mathrm{if}~ k \geq 2,
\end{array} \right .
\end{equation}
where $\eta_k \in [ \eta_{min} ,\eta_{max}] $, $ \eta_{min} \in [0,1] $, $ \eta_{max} \in [\eta_{min},1] $. More recently, {\sc Amini} et al. in \cite{AmiAN} proposed the nonmonotone term 
\begin{equation}\label{e.Ami}
R_k=\eta_k f_{l(k)}+(1-\eta_k)f_k,\end{equation}
where $0\leq\eta_{min}\leq \eta_{max}\leq1$ and
$\eta_{k}\in[\eta_{min},\eta_{max}]$. In all cases it was proved that the schemes are globally convergent and enjoy the better performance compared with monotone ones. 

At the same importance of using monmonotone strategies for inexact line search techniques, the combination of trust-region methods with nonmonotone strategies is interesting. Historically, the first nonmonotone trust-region method was proposed in 1993 by {\sc Deng} et al. in \cite{DenXZ} for unconstrained optimization. Under some classical assumptions, the global convergence and the local superlinear convergence rate were established. Nonmonotone trust-region methods were also studied by several authors such as {\sc Toint} \cite{Toi2}, {\sc Xiao \& Zho} \cite{XiaZ}, {\sc Xiao \& Chu} \cite{XiaC}, {\sc Zhou \& Xiao} \cite{ZhoX}, {\sc Ahookhosh \& Amini} \cite{AhoA2}, {\sc Amini \& Ahookhosh} \cite{AmiA}, and {\sc Mo} et al. \cite{MoLY}. Recently, {\sc Ahookhosh \& Amini} in \cite{AhoA1} and {\sc Ahookhosh} et al. in \cite{AhoAP} proposed two nonmonotone trust-region methods using the nonmonotone term (\ref{e.Ami}). Theoretical results were reported, and numerical results  showed the efficiency of the proposed nonmonotone methods.\\\\
{\bf Content.} In this paper we propose a trust-region method equipped with two novel nonmonotone terms. More precisely, we first establish two nonmonotone terms and then combine them with Algorithm 1 to construct two nonmonotone trust-region algorithms. If $k \geq N$, the new nonmonotone terms are defined by a convex combination of the last $N$ successful function values, and if $k < N$, either a convex combination of $k$ successful function values or $f_{l(k)}$ is used. The global convergence to first- and second-order stationary points is established on some classical assumptions. Moreover, local superlinear and quadratic convergence rates for the proposed methods are studied. Numerical results regarding experiments on some highly nonlinear problems and on 112 unconstrained test problems from the \textsf{CUTEst} test collection \cite{GolOT} are reported indicating the efficiency of the proposed nonmonotone terms.

The remainder of paper is organized as follow. In Section 2 we propose new nonmonotone terms and their combination with the trust-region framework. The global convergence of the proposed methods are given in Section 3. Numerical results are reported in Section 4. Finally, some conclusions are given in Section 5.

% ################################################
% ################################################
\section {Novel nonmonotone terms and algorithm} \label{sec1}
In this section we first present two novel nonmonotone
terms and then combine them into trust-region framework to
introduce two nonmonotone trust-region algorithms
for solving the unconstrained optimization problem (\ref{e.func}).
 
We first assume that $k$ denotes the current iteration and  $N \in \mathbb{N}$ is a constant. The main idea is to construct a nonmonotone term determined by a convex combination of the last $k$ successful function values if $k < N$ and by a convex combination of the last $N$ successful function values if $k \geq N$. In the other words, we construct new terms using function values collected in the set 
\begin{equation} 
 \mathcal{F}_k := \left\{ 
\begin{array}{ll} 
\{f_0,f_1, \cdots, f_k\}             &~~\mathrm{if}~ k < N,\\ 
\{f_{k-N+1},f_{k-N+2}, \cdots, f_k\} &~~\mathrm{if}~ k \geq N,
\end{array} \right.
\end{equation} 
which should be updated in each iteration. To this end, 
motivated by the term (\ref{e.dk}), we construct $\overline{T}_k$ using the subsequent procedure    
\begin{equation*} 
\left\{ 
\begin{array}{lll} 
\overline{T}_0     & = f_0                                                                &  ~~\mathrm{if}~k = 0,\\ 
\overline{T}_1     & = (1-\eta_{0})f_1 + \eta_{0} f_0                                     &  ~~\mathrm{if}~k = 1,\\ 
\overline{T}_2     & = (1-\eta_{1})f_2 + \eta_{1}(1-\eta_{0}) f_1 + \eta_{1}\eta_{0} f_0  &  ~~\mathrm{if}~k = 2,\\ 
\vdots            & \vdots                                                                &  ~~  \vdots\\ 
\overline{T}_{N-1} & = (1-\eta_{N - 2})f_{N - 1} + \eta_{N - 2}(1-\eta_{N - 3}) f_{N-2}  
             + \cdots + \eta_{N - 2} \cdots \eta_{0} f_0                                   &  ~~\mathrm{if}~k = N  - 1,\\ 
\overline{T}_N     & = (1-\eta_{N - 1})f_N + \eta_{N - 1}(1-\eta_{N - 2}) f_{N-1}  
            +  \cdots  +  \eta_{N  -  1} \cdots \eta_{0} f_0                               &  ~~\mathrm{if}~k = N,
\end{array} 
\right .
\end{equation*} 
where $\eta_i \in [0,1)$, for $i = 1 ,  2 ,  \cdots ,  N$, are some weight parameters. Hence the new term is generated by
\begin{equation}\label{e.tk} 
\overline{T}_k  : = \left\{ 
\begin{array}{ll} 
 (1-\eta_{k-1}) f_k  +  \eta_{k-1} \overline{T}_{k-1}       &~~ \mathrm{if}~ k < N,\\ 
 (1-\eta_{k-1})~f_k  +  \eta_{k-1}(1-\eta_{k-2}) ~ f_{k-1}  
+  \cdots + \eta_{k-1} \cdots \eta_{k-N} ~  f_{k-N}        &~~ \mathrm{if}~ k \geq N,
\end{array} \right.
\end{equation} 
where $ \overline{T}_0 = f_0$ and $\eta_i \in [0,1)$ for $i = 1, 2 , \cdots ,  k$. To show that $\overline{T}_k$ is a convex combination of the collected function values $\mathcal{F}_k$, it is enough to show that the summation of multipliers are equal to unity. For $k \geq N$, the definition for $\overline{T}_k$ implies  
\begin{equation}\label{e.conv} 
 (1-\eta_{k  -  1})  +  \eta_{k  -  1}(1-\eta_{k  -  2})  +  \cdots  +  \eta_{k  -  1} \cdots \eta_{k  -  N  -  1} (1-\eta_{k  -  N}) +  \eta_{k  -  1} \cdots \eta_{k  -  N} = 1 
\end{equation} 
For $k < N$, a similar summation of the last $k$ multipliers is equal to one. Therefore, the generated term $\overline{T}_k$ is a convex combination of the elements of $\mathcal{F}_k$.

The procedure of defining $\overline{T}_k$ clearly implies that the set $\mathcal{F}_k$ should be updated and saved in each iteration. Moreover, $N(N  +  1)/2$ multiplications is required to compute $ \overline{T}_k$. To avoid saving $\mathcal{F}_k$ and decrease the number of multiplications, we derive a recursive formula for (\ref{e.tk}). From the definition of $ \overline{T}_k$, for $k \geq N$, it follows that 
\begin{equation*} 
\begin{split} 
 \overline{T}_k  -  \eta_{k  -  1}  \overline{T}_{k  -  1} & = (1-\eta_{k  -  1})~f_k  +  \eta_{k  -  1}(1-\eta_{k  -  2}) ~  f_{k  -  1}  
+  \cdots  +  \eta_{k  -  1} \cdots \eta_{k  -  N} ~  f_{k  -  N}\\ 
 & -  \eta_{k  -  1} (1-\eta_{k  -  2})~f_{k  -  1}  -  \cdots  -  \eta_{k  -  1} \cdots (1-\eta_{k  -  N  -1})~  f_{k  -  N}  -  \eta_{k  -  1} \eta_{k  -  2} \cdots \eta_{k  -  N  -  1} ~  f_{k  -  N  - 1}\\ 
 &= (1-\eta_{k  -  1})~f_k  +  \eta_{k  -  1} \eta_{k  -  2} \cdots \eta_{k  -  N  -  1} ~  (f_{k  -  N}  -  f_{k  -  N  - 1})\\ 
 &= (1-\eta_{k  -  1})~f_k  +  \xi_k ~  (f_{k  -  N}  -  f_{k  -  N  - 1}) 
\end{split} 
\end{equation*} 
 where  
 $\xi_k  : = \eta_{k  -  1} \eta_{k  -  2} \cdots \eta_{k  -  N  -  1}$ . For $k \geq N$, this equation leads to
\begin{equation}\label{e.tk0} 
 \overline{T}_k = (1-\eta_{k  -  1})~f_k  +  \eta_{k  -  1}  \overline{T}_{k  -  1}   +\xi_k~   (f_{k  -  N}  -  f_{k  -  N  -1}),
\end{equation} 
which requirs to save only $f_{k-N}$ and $f_{k-N-1}$ and only needs three multiplications. Moreover, the definition of $\xi_k$ implies  
\begin{equation}\label{e.xik} 
 \xi_k = \eta_{k  -  1} \eta_{k  -  2} \cdots \eta_{k  -  N  -  1} = \frac{\eta_{k  -  1}}{\eta_{k  -  N  -  2}} \eta_{k  -  2} \eta_{k  -  3} \cdots \eta_{k  -  N  -  2} = \frac{\eta_{k  -  1}}{\eta_{k  -  N  -  2}} \xi_{k-1} .
\end{equation} 
If $\xi_k$ is recursively updated by (\ref{e.xik}), (\ref{e.tk}), and (\ref{e.tk0}), a new nonmonotone term is defined by 
\begin{equation}\label{e.tk1} 
 T_k  : = \left\{ 
\begin{array}{ll} 
 f_k  +  \eta_{k  -  1} ( \overline{T}_k  -  f_k)   &~~ \mathrm{if}~ k < N,\\ 
 \max \left\{\overline{T}_{k} ,  f_k \right\}   &~~ \mathrm{if}~ k \geq N,
\end{array} \right.
\end{equation} 
where the max term is added to guarantee $T_k \geq f_k$ .

As discussed in Section 1, nonmonotone schemes perform better when they use stronger nonmonotone terms far away from the optimizer and weaker one close to it. This  motivate us to consider a new version of the derived nonmonotone term by using $f_{l(k)}$ in cases that $k < N$. More precisely, the second nonmonotone term is defined by   
 \begin{equation}\label{e.tk2} 
 T_k = \left\{ 
  \begin{array}{ll} 
f_{l(k)} &~~ \mathrm{if}~ k < N,\\ 
\max \left\{\overline{T}_{k},  f_k \right\}  &~~ \mathrm{if}~ k \geq N,
 \end{array} \right .
 \end{equation} 
where $\xi_k$ is defined by (\ref{e.xik}). It is clear that the new term uses a stronger term $f_{l(k)}$ defined by (\ref{e.grip}) for first $k < N$ iterations and then employs the relaxed convex term proposed above. 

Now, to employ the proposed nonmonotone terms in the trust-region framework, it is enough to replace the ratio $r_k$ (\ref{e.rk}) by the nonmonotone ratio
\begin{equation} \label{e.rkhat}
\widehat{r}_k = \frac{T_k-f(x_k+d)}{q_k(0)-q_k(d)},
\end{equation}
where $T_k$ is defined by (\ref{e.tk1}) or (\ref{e.tk2}).
Hence in trust-region framework we replace (\ref{e.rk})
by (\ref{e.rkhat}). 
Notice that if $\widehat{r}_k \geq \mu_1$, the,
\[
T_k-f_{k+1} \geq \mu_1(q_k(0)-q_k(d_k)) \geq 0.
\]
This implies that $f_{k+1}$ can be larger than $f_k$, however, the elements of $\{f_k\}$ cannot arbitrarily increase, and the maximum increase is controlled by the nonmonotone term $T_k$. Moreover, the definitions (\ref{e.tk1}) and (\ref{e.tk2}) imply that $\widehat{r}_k \geq r_k$ increasing the possibility of attaining larger steps for nonmonotone schemes compared with monotone ones. 

The above-mentioned discussion leads to the following nonmonotone trust-region algorithm:

\vspace{6mm}
%%%%%%%%%%%%%%%%%%%%%%%%%%%%%%%%%
\begin{algorithm}[H] \label{NTTR}
\DontPrintSemicolon % Some LaTeX compilers require you to use \dontprintsemicolon instead
\KwIn{$x_0 \in \mathbb{R}^n$, $B_0 \in \mathbb{R}^{n \times n}$, $k_{max}$; $0<\mu_1 \leq \mu_2 \geq 1$, $0 < \rho_1 \leq 1 \leq \rho_2 \geq 1$,
$\varepsilon >0$;}
\KwOut{$x_b$; $f_b$;}
\Begin{
    $\delta_0\leftarrow \|g_0\|$;~ $k \leftarrow 0$;\;
    \While {$\|g_k\| \geq \varepsilon~~ \&~~ k \leq k_{max}$}{
        solve the subproblem (\ref{e.sub}) to specify $d_k$;\; 
        $\widehat{x}_k \leftarrow x_k + d_k$;~ compute $f(\hat{x}_k)$;\;
        determine $\widehat{r}_k$ using (\ref{e.rkhat});\;
        if {$\widehat{r}_k < \mu_1$}{ 
        \While {$\widehat{r}_k < \mu_1$}{ 
            $\delta_k \leftarrow \rho_1 \delta_k $;\;
            solve the subproblem (\ref{e.sub}) to specify $d_k$;\;
            $\widehat{x}_k \leftarrow x_k + d_k$; ~ compute $f(\hat{x}_k)$\;
            determine $\widehat{r}_k$ using (\ref{e.rkhat});\;
        }}
        $x_{k+1} \leftarrow \widehat{x}_k$;\;
        \If{$\widehat{r}_k \geq \mu_2$}{
        $\delta_{k + 1} \leftarrow \rho_2 \delta_k $;
        }
        update $B_{k+1}$;~ $k \leftarrow k+1$;\;
        update $T_{k+1}$;\;
        update $\eta_{k+1}$;\;
    }
    $x_b \leftarrow x_k$;~ $f_b \leftarrow f_k$;
}
\caption{ {\bf NMTR} (nonmonotone traditional trust-region algorithm)}
\label{algo:max}
\end{algorithm}
%%%%%%%%%%%%%%%

\vspace{5mm}
In Algorithm 2, if $\widehat{r}_k \geq \mu_1$ (Line 7), it is called a
{\bf \emph{successful}} iteration and if $\widehat{r}_k \geq \mu_2$ 
(Line 14), it is called a {\bf \emph{very successful}} iteration. In addition, in the algorithm, the loop started from Line 3 to Line 20 is called the {\bf \emph{outer cycle}}, and the loop started from Line 7 to Line 12 is called the {\bf \emph{inner cycle}}.

% ################################################
% ################################################
\section {Convergence analysis}
This section concerns with the global convergence 
to first- and second-order stationary points of the 
sequence $\{x_k\}$ generated by Algorithm 2. More
precisely, we intend to prove that all limit point 
$x^*$ of the sequence $\{x_k\}$ satisfy the 
condition $g(x^*)=0$, and there exists a point $x^*$ satisfying $g(x^*)=0$ where $H(x^*)$ is positive semidefinite. Furthermore, we show that Algorithm 2 is well-defined, which means that the inner cycle of the algorithm will be leaved after a finite number internal iterations, and then prove its global convergence. Moreover, local superlinear and quadratic convergence rates are investigated under some classical assumptions.

To prove the global convergence of the sequence $\{x_k\}$
generated by Algorithm 2, we require to make the following
assumptions:
\\\\
{\bf (H1)} The objective function $f$ is continuously 
differentiable and has a lower bound on the upper level
set $L(x_0) = \{x \in \mathbb{R}^n \mid f(x) \leq f(x_0)\}$.\\
{\bf (H2)} The sequence $\{B_k\}$ is uniformly bounded, i.e., 
there exists a constant $M>0$ such that
\[
\|B_k\| \leq M,
\]
for all $k \in \mathbb{N}$.\\
{\bf (H3)} There exists a constant $c>0$ such that the trial
step $d_k$ satisfies $\|d_k\| \leq c\|g_k\|$.\\

We also assume that the decrease on the model $q_k$ is at 
least as much as a fraction of the decrease obtained by the Cauchy point guaranteeing that there exists a constant $\beta\in(0,1)$ such 
that
\begin{equation}\label{e.ins}
q_k(0)-q_k(d_k)\geq \beta \|g_k\|~ \min\left\{\delta_k,
\frac{\|g_k\|}{\|B_k\|} \right\},
\end{equation}
for all $k$. This condition is called the sufficient 
reduction condition. Inequality (\ref{e.ins}) implies that 
$d_k \neq 0$ whenever $g_k \neq 0$. It is noticeable that 
there are several schemes that can solve the the trust-region
subproblem (\ref{e.sub}) such that (\ref{e.ine}) is valid,
see, for example, \cite{ConGT, NocW}.  

%%%%%%%%%%%%%%%%%%%%%%%%%
\begin{lem} \label{l.dif}
Suppose that sequence $\{x_k\}$ is generated by Algorithm 2, 
then 
\begin{equation*}
|f_k-f(x_k+d_k)-(q_k(0)-q_k(d_k))|\leq O(\|d_k\|^2).
\end{equation*}
\end{lem}

%%%%%%%%%%%%%
\begin{proof} 
The proof can be found in \cite{ChaPLP}. \qed 
\end{proof}
%%%%%%%%%%%

%%%%%%%%%%%%%%%%%%%%%%%%
\begin{lem} \label{lem1}  
Suppose that the sequence $\{x_{k}\}$ is generated by Algorithm 1, 
then we get 
\begin{equation}\label{e.ine}  
 f_{k}\leq T_{k}\leq f_{l(k)},
\end{equation} 
for all $k \in \mathbb{N} \cup \{0\}$.
\end{lem} 

%%%%%%%%%%%%% 
\begin{proof} 
For $k \leq N$, we consider two cases: (i) $T_k$ 
is defined by (\ref{e.tk1}); (ii) $T_k$ is defined by (\ref{e.tk2}). 
In Case (i) Lemma 2.1 in \cite{AhoAB}, $f_i \leq f_{l(k)}$, for  
$i = 0, 1, \cdots k $, and the fact that summation of multipliers 
in $T_k$ equal to one give the result. Case (ii) is evident from 
(\ref{e.tk2}).

For $k \geq N$, if $T_k = f_k$, the result is evident. 
Otherwise, since
\begin{equation}\label{e.conv} 
 (1-\eta_{k  -  1})  +  \eta_{k  -  1}(1-\eta_{k  -  2})  +  
 \cdots  +  \eta_{k  -  1} \cdots \eta_{k  -  N  -  1} 
 (1-\eta_{k  -  N}) +  \eta_{k  -  1} \cdots \eta_{k  -  N} = 1,
\end{equation} 
the fact that $f_i \leq f_{l(k)}$, for $i = k - N + 1, \cdots, k$, 
and (\ref{e.tk}) imply 
 \[ 
 \begin{split} 
 f_k \leq T_k &= (1-\eta_{k  -  1})~f_k  +  \eta_{k  -  1}
 (1-\eta_{k  -  2}) ~  f_{k-1}  
+  \cdots  +  \eta_{k  -  1} \cdots \eta_{k  -  N} ~  
f_{k  -  N}\\ 
 & \leq [(1-\eta_{k  -  1})  +  \eta_{k  -  1}
 (1-\eta_{k  -  2})  
+  \cdots  +  \eta_{k  -  1} \cdots \eta_{k  -  N}] 
f_{l(k)} = f_{l(k)},
 \end{split} 
 \] 
giving the result. \qed
\end{proof}
%%%%%%%%%%%

%%%%%%%%%%%%%%%%%%%%%%%%%
\begin{lem} \label{e.flk}
Suppose that sequence $\{x_k\}$ is generated by Algorithm 2, 
then the sequence $\{f_{l(k)}\}$ is decreasing.
\end{lem}

%%%%%%%%%%%%%
\begin{proof}
The condition (\ref{e.ine}) implies that $T_k \leq f_{l(k)}$.
If $x_{k+1}$ is accepted by Algorithm 2, then 
\begin{equation*}
\frac{f_{l(k)}-f(x_k+d_k)}{q_k(0)-q_k(d_k)}\geq
\frac{T_k-f(x_k+d_k)}{q_k(0)-q_k(d_k)}\geq \mu_1,
\end{equation*}
leading to
\begin{equation*}
f_{l(k)}-f(x_k+d_k)\geq \mu_1 (q_k(0)-q_k(d_k)) \geq 0,
~~~ \mathrm{for~all} ~k \in \mathbb{N},
\end{equation*}
implying
\begin{equation}\label{e:12}
f_{l(k)}\geq f_{k+1},~~~ \mathrm{for~all} ~k \in \mathbb{N}.
\end{equation}
Now, if $k\geq N$, by using $m(k+1) \leq m(k)+1$ and
(\ref{e:12}), we get
\begin{equation*}
f_{l(k+1)}= \max_{0\leq j\leq m(k+1)}\{f_{k-j+1}\} \leq
\max_{0\leq j\leq m(k)+1}\{f_{k-j+1}\} = \max
\{f_{l(k)},f_{k+1}\}\leq f_{l(k)}.
\end{equation*}
For $k<N$, it is obvious that $m(k)=k$. Since, for any $k$, $f_k\leq
f_0$, it is clear that $f_{l(k)}=f_0$. Therefore, in both cases,
the sequence $\{f_{l(k)}\}$ is decreasing. \qed
\end{proof}
%%%%%%%%%%%

%%%%%%%%%%%%%%%%%%%%%%%%%
\begin{lem} \label{e.inv}
Suppose that (H1) holds and the sequence $\{x_k\}$ is generated 
by Algorithm 2, then $L(x_0)$ involves $\{x_k\}$. 
\end{lem}

%%%%%%%%%%%%%
\begin{proof}
The definition of $T_k$ indicates that $T_0=f_0$. By
induction, we assume that $x_i \in L(x_0)$, for all
$i=1,2,\cdots,k$, and then prove that $x_{k+1} \in L(x_0)$. 
From (\ref{e.ine}), we get
\begin{equation*} 
f_{k+1} \leq T_{k+1} \leq f_{l(k+1)} \leq f_{l(k)} \leq f_0, 
\end{equation*}
implying that $L(x_0)$ involves the sequence $\{x_k\}$.\qed
\end{proof}
%%%%%%%%%%%

%%%%%%%%%%%%%%%%%%%%%%%%%%%%
\begin{cor} \label{e.conflk}
Suppose that (H1) holds and the sequence $\{x_k\}$ is generated
by Algorithm 2. Then the sequence $\{f_{l(k)}\}$ is convergent.
\end{cor}

%%%%%%%%%%%%%
\begin{proof}
The assumption (H1) and Lemma \ref{e.flk} imply that there
exists a constant $\lambda$ such that 
\begin{equation*}
\lambda\leq
f_{k+n}\leq f_{l(k+n)}\leq \cdots \leq f_{l(k+1)}\leq f_{l(k)},
\end{equation*}
for all $n \in \mathbb{N}$. This implies that the sequence 
$\{f_{l(k)}\}$ is convergent. \qed
\end{proof}
%%%%%%%%%%%

%%%%%%%%%%%%%%%%%%%%%%%%
\begin{lem} \label{lem7} 
Suppose that (H1)-(H3) hold and the sequence $\{x_k\}$ is
generated by Algorithm 2, then 
\begin{equation}\label{e:14} 
\lim_{k\rightarrow \infty}f(x_{l(k)})=\lim_{k\rightarrow \infty}f_k.
\end{equation}
\end{lem}

%%%%%%%%%%%%%
\begin{proof}
The condition (\ref{e.ine}) and Lemma 7 of \cite{AhoA1} imply that the result is valid. \qed 
\end{proof}
%%%%%%%%%%%

%%%%%%%%%%%%%%%%%%%%%%%%%%%
\begin{cor} \label{c.limtk}
Suppose (H1)-(H3) hold and the sequence $\{x_k\}$ is generated by
Algorithm 2, then we
\begin{equation}\label{e.20} 
\lim_{k\rightarrow \infty}T_k=\lim_{k\rightarrow \infty}f_k.
\end{equation}
\end{cor}

%%%%%%%%%%%%%
\begin{proof}
From (\ref{e.ine}) and Lemma \ref{lem7}, the result is obtained.
\qed 
\end{proof}
%%%%%%%%%%%

%%%%%%%%%%%%%%%%%%%%%%%%%%
\begin{lem} \label{l.vers}
Suppose that (H1) and (H2) hold, and the sequence $\{x_k\}$ is generated by Algorithm 2. Then if $\|g_k\| \geq \varepsilon>0$, we have\\
(i) The inner cycle of Algorithm 2 is well-defined;\\
(ii) For any $k$, there exists a nonnegative integer $p$ 
such that $x_{k+p+1}$ is a very successful iteration.
\end{lem}

%%%%%%%%%%%%%
\begin{proof} 
(i) Let $t_k$ denotes the internal iteration counter in  step $k$, and $d_k^{t_k}$ and $\delta_k^{t_k}$ respectively show the solution of the subproblem (\ref{e.sub}) and the corresponding trust-region radius in the internal iteration $t_k$. The fact that $\|g_k\| \geq \varepsilon>0$, (H2), 
and (\ref{e.ins}) imply
\begin{equation}\label{e:23}
q_k(0)-q_k(d_k^{t_k})\geq \beta \|g_k\|~\min\left\{\delta_k^{t_k},
\frac{\|g_k\|}{\|B_k\|} \right\} \geq \beta \varepsilon~
\min \left\{\delta_k^{t_k}, \frac{\varepsilon}{M} \right\}.
\end{equation}
Then Line 8 of Algorithm 2 implies
\begin{equation*}
\lim_{t_k \rightarrow \infty} \delta_k^{t_k}=0.
\end{equation*}
From This, Lemma \ref{l.dif}, and (\ref{e:233}), we obtain
\begin{equation*}\begin{split}
|r_k-1|&=\left|\frac{f_k-f(x_k+d_k^{t_k})}
{q_k(0)-q_k(d_k^{t_k})}-1\right|
=\left|\frac{f_k-f(x_k+d_k^{t_k})-(q_k(0)-q_k(d_k^{t_k}))}{q_k(0)-q_k(d_k^{t_k})}\right|\\
&\leq \frac{O(\|d_k^{t_k}\|^2)}{\beta \varepsilon~
\min\left\{\delta_k^{t_k}, \varepsilon/M \right\}}\leq
\frac{O((\delta_k^{t_k})^2)}{\beta \varepsilon~ \min\left\{\delta_k^{t_k},
\varepsilon/M \right\}}\rightarrow 0 ~~~
(t_k\rightarrow\infty),
\end{split}\end{equation*}
implying that there exists a positive integer $k_0$ such that for $k\geq k_0$
we have $r_k \geq \mu_1$. This and (\ref{e.ine}) lead to 
\begin{equation*}
\widehat{r}_k = \frac{T_k-f(x_k+d_k^{t_k})}{q_k(0)-q_k(d_k^{t_k})}
\geq\frac{f_k-f(x_k+d_k^{t_k})}{q_k(0)-q_k(d_k^{t_k})}
\geq \mu_1,
\end{equation*}
implying that the inner cycle is well-defined.

(ii) Assume that there exists a positive integer
$k$ such that for an arbitrary positive integer $p$ the point
$x_{k+p+1}$ is not very successful. Hence, for any constant
$p=0,1,2,\cdots$, we get
\[
\widehat{r}_{k+p}<\mu_2.
\]
The fact that $\|g_k\| \geq \varepsilon>0$, (H2), 
and (\ref{e.ins}) imply
\begin{equation}\label{e:233}
\begin{split}
T_{k+p} - f(x_{k+p} + d_{k+p}) &\geq \mu_1 (q_{k+p}(0)-q_{k+p}(d_{k+p}))\geq \beta \mu_1 \|g_{k+p}\|~\min\left\{\delta_{k+p},
\frac{\|g_{k+p}\|}{\|B_{k+p}\|} \right\} \\
&\geq \beta \mu_1 \varepsilon~
\min \left\{\delta_{k+p}, \frac{\varepsilon}{M} \right\}.
\end{split}
\end{equation}
By using (\ref{e.20}) and (\ref{e:233}), we can write
\begin{equation}\label{e:22}
\lim_{p \rightarrow \infty} \delta_{k+p}=0.
\end{equation}
From Lemma \ref{l.dif}, (\ref{e:22}), and (\ref{e:23}), we obtain
\begin{equation*}\begin{split}
|r_{k+p}-1|&=\left|\frac{f(x_{k+p})-f(x_{k+p}+d_{k+p})}
{q_{k+p}(0)-q_{k+p}(d_{k+p})}-1\right|\\
&=\left|\frac{f(x_{k+p})-f(x_{k+p}+d_{k+p})-(q_{k+p}(0)-q_{k+p}
(d_{k+p}))}{q_{k+p}(0)-q_{k+p}(d_{k+p})}\right|\\
&\leq \frac{O(\|d_{k+p}\|^2)}{\beta \varepsilon~
\min\left\{\delta_{k+p}, \varepsilon/M \right\}}\leq
\frac{O(\delta_{k+p}^2)}{\beta \varepsilon~ \min\left\{\delta_{k+p},
\varepsilon/M \right\}}\rightarrow 0 ~~~
(p\rightarrow\infty).
\end{split}\end{equation*}
Then, for a sufficiently large $p$, we get $r_{k+p} \geq \mu_2$
leading to
\begin{equation*}\frac{T_{k+p}-f(x_{k+p}+d_{k+p})}{q_{k+p}(0)-q_{k+p}(d_{k+p})}
\geq\frac{f(x_{k+p})-f(x_{k+p}+d_{k+p})}{q_{k+p}(0)-q_{k+p}(d_{k+p})}
\geq \mu_2.
\end{equation*}
implying $\widehat{r}_{k+p}\geq\mu_2$, for a sufficiently
large $p$. This contradicts with assumption $\widehat{r}_{k+p}<\mu_2$ giving the result.
\qed
\end{proof}
%%%%%%%%%%%

Lemma \ref{l.vers}(i) implies that the inner cycle will be leaved after a finite number of internal iterations, and Lemma \ref{l.vers}(ii) implies that if the current iteration is not a first-order stationary point, then at least there exists a very successful iteration point, i.e., the trust-region radius $\delta_k$ can be enlarged. The next result gives the global convergence of the sequence $\{x_k\}$ of Algorithm 2.

%%%%%%%%%%%%%%%%%%%%%%%%%%
\begin{thm} \label{t.gloc}
Suppose that (H1) and (H2) hold, and suppose the sequence 
$\{x_k\}$ is generated by Algorithm 2. Then
\begin{equation}\label{e:24}\liminf_{k \rightarrow \infty} \|g_k\|=0.\end{equation}
\end{thm}

%%%%%%%%%%%%%
\begin{proof}
We consider two cases: 
(i) Algorithm 2 has finitely many very successful iterations;
(ii) Algorithm 2 has infinitely many very successful iterations.

In Case 1, we suppose that $k_0$ be the largest index of
very successful iterations. If $\|g_{k_0+1}\|>0$, 
then Lemma \ref{l.vers}(ii) implies that there exist a very successful iteration with larger index than $k_0$.
This is a contradiction to the definition of $k_0$.

In Case 2, by contradiction, we assume that there exist constants
$\varepsilon>0$ and $K>0$ such that
\begin{equation}\label{e:25} 
\|g_k\|\geq \varepsilon,
\end{equation}
for all $k\geq K$. If $x_{k+1}$ is a successful iteration 
and $k \geq K$, then by using
(H2), (\ref{e.ins}), and (\ref{e:25}), we get
\begin{equation}\begin{split}\label{e:26}
T_k-f(x_k+d_k) & \geq \mu_1(q_k(0)-q_k(d_k))\\
&\geq \beta \mu_1 \|g_k\|~ \min\left\{\delta_k,
\frac{\|g_k\|}{\|B_k\|}\right\} \geq \beta \mu_1 \varepsilon~
\min\left\{\delta_k, \frac{\varepsilon}{M} \right\} \geq 0.
\end{split}\end{equation}
It follows from this inequality and (\ref{e.20}) that
\begin{equation}\label{e:27}
\lim_{k \rightarrow \infty} \delta_k=0.
\end{equation}
Since Algorithm 2 has infinitely many very successful iterations, 
then Lemma \ref{l.vers}(ii) and (\ref{e:25}) imply that the sequence 
$\{x_k\}$ involves infinitely many very successful iterations in 
which the trust-region is enlarged, which is a contradiction
with (\ref{e:27}). This implies the result is valid. \qed
\end{proof}
%%%%%%%%%%%

%%%%%%%%%%%%%%%%%%%%%%%%%%
\begin{thm} \label{t.sgolc}
Suppose that (H1) and (H2) hold, and the sequence $\{x_k\}$ is generated by Algorithm 2. Then 
\begin{equation}\label{e:28}
\lim_{k \rightarrow \infty} \|g_k\|=0.
\end{equation}
Moreover, there is no limit point of the sequence $\{x_k\}$ 
to be a local maximizer of $f$.
\end{thm}

%%%%%%%%%%%%%
\begin{proof}
By contradiction, we assume $\lim_{k \rightarrow \infty} \|g_k\|\neq0$. Hence there exists $\varepsilon>0$ and an infinite subsequence of $\{x_k\}$, indexed by $\{t_i\}$, such that
\begin{equation} \label{e:29} 
\|g_{t_i}\| \geq 2\varepsilon>0, 
\end{equation}
for all $i \in \mathbb{N}$. Theorem \ref{t.gloc} ensures the existence, for each $t_i$, a first successful
iteration $r(t_i)>t_i$ such that $\|g_{r(t_i)}\| < \varepsilon$. We denote $r_i=r(t_i)$. Hence there exists another subsequence, indexed by $\{r_i\}$, such that
\begin{equation}\label{e:30} 
\|g_k\|\geq \varepsilon~~ \mathrm{for}~~ t_i \leq k < r_i,~~~ \|g_{r_i}\|<\varepsilon.
\end{equation}
We now restrict our attention to the sequence of successful
iterations whose indices are in the set 
\[
\kappa = \{k \in \mathbb{N}\ |\ t_i \leq k < r_i \}.
\]
Using (\ref{e:30}), for every $k \in \kappa$, (\ref{e:26}) holds. It follows from (\ref{e.20}) and (\ref{e:26}) that 
\begin{equation}\label{e:31}
\lim_{k \rightarrow \infty} \delta_k=0, 
\end{equation}
for $k \in \kappa$. Now, using (H2), (\ref{e.ins}), and $\|g_k\| \geq \varepsilon$, the condition (\ref{e:23}) holds, for $k \in \kappa$. This, Lemma \ref{l.dif}, and (\ref{e:31}) lead to
\begin{equation*}\begin{split}
|r_k-1|&=\left|\frac{f_k-f(x_k+d_k)}{q_k(0)-q_k(d_k)}-1\right|=\left|\frac{f_k-f(x_k+d_k)-(q_k(0)-q_k(d_k))}{q_k(0)-q_k(d_k)}\right|\\
&\leq \frac{O(\|d_k\|^2)}{\beta \varepsilon~ \min\left\{\delta_k,
\varepsilon/M \right\}}\leq \frac{O(\delta_k^2)}{\beta
\varepsilon \delta_k }\rightarrow 0 ~~~ (k\rightarrow\infty,~ k\in\kappa).
\end{split}\end{equation*}
Thus, for a sufficiently large $k+1 \in \kappa$, we get
\begin{equation}\begin{split}\label{e:32}
f_k-f(x_k+d_k) & \geq \mu_1(q_k(0)-q_k(d_k))\\
&\geq \beta \mu_1 \|g_k\| \min\left\{\delta_k,
\frac{\|g_k\|}{\|B_k\|}\right\} \geq \beta \mu_1 \varepsilon~
\min\left\{\delta_k, \frac{\varepsilon}{M} \right\}.
\end{split}\end{equation}
The condition (\ref{e:31}) implies that $\delta_k \leq \varepsilon/M$. Hence, for a sufficiently large $k \in \kappa$, we obtain
\begin{equation}\label{e:33}\delta_k \leq \frac{1}{\beta \mu_1}(f_k-f_{k+1}).\end{equation}
Then (\ref{e.ine}) and (\ref{e:33}) imply
\begin{equation}\label{e:34}
\|x_{t_i}-x_{r_i}\| \leq \sum_{j \in \kappa,
j=t_i}^{r_i-1}\|x_{j}-x_{j+1}\| \leq \sum_{j \in \kappa,
j=t_i}^{r_i-1} \delta_j \leq \frac{1}{\beta
\mu_1}(f_{t_i}-f_{r_i}) \leq \frac{1}{\beta
\mu_1}(T_{t_i}-f_{r_i}),
\end{equation}
for a sufficiently large $i$. Now, Corollary 8 implies 
\[
0 \leq \lim_{i \rightarrow \infty} \|x_{t_i}-x_{r_i}\| \leq  \lim_{i \rightarrow \infty} \frac{1}{\beta
\mu_1}(T_{t_i}-f_{r_i}) = 0,
\]
leading to
\[
\lim_{i \rightarrow \infty} \|x_{t_i}-x_{r_i}\| = 0.
\]
Since the gradient is continuous, we get
\begin{equation}\label{e:35}
\lim_{i \rightarrow \infty} \|g_{t_i}-g_{r_i}\|=0.\end{equation}
In view of the definitions of $\{t_i\}$ and $\{r_i\}$, it is impossible, guaranteeing $\|g_{t_i}-g_{r_i}\| \geq \varepsilon$. Therefore, there is no subsequence that satisfies (\ref{e:29}) giving the result. 

To observe there is no limit point of the sequence $\{x_k\}$ 
to be a local maximizer of $f$, see \cite{GriLL1}. \qed
\end{proof}
%%%%%%%%%%%

The next result gives the global convergence of the sequence generated by Algorithm 2 to second-order stationary points. To this end, similar to \cite{DenXZ}, an additional assumption is needed: \\\\
{\bf(H4)} If $\lambda_{\min}(B_k)$ represents the smallest
eigenvalue of the symmetric matrix $B_k$, then there exists a positive scalar $c_3$ such that
\begin{equation*}
q_k(0)-q_k(d_k) \geq c_3 \lambda_{\min}(B_k) \delta^2.
\end{equation*}

%%%%%%%%%%%%%%%%%%%%%%%%%%
\begin{thm} \label{t.scon}
Suppose that $f$ is twice continuously differentiable
and also suppose that (H1)--(H4) hold. Then there exists a limit point $x^*$ of the sequence $\{x_k\}$ generated by Algorithm 2 such that $\nabla^2 f(x^*)$ is positive semidefinite.
\end{thm}

%%%%%%%%%%%%%
\begin{proof}
The proof is similar to Theorem 3.4 in \cite{DenXZ}. \qed
\end{proof}

The next two results show that Algorithm 2 can be reduced to 
quasi-Newton or Newton methods, where the sequence
$\{x_k\}$ generated by these schemes can attain  
local superlinear and quadratic convergence rates
under some conditions, respectively. 

%%%%%%%%%%%%%%%%%%%%%%%%%
\begin{thm} \label{t.sup}
Suppose that (H1)--(H3) hold, and also suppose that the sequence $\{x_k\}$ is generated by Algorithm 2 converges to $x^*$, $\|d_k\|=\|-B_k^{-1}g_k\| \leq \delta_k$, $H(x)=\nabla^2 f(x)$ is continuous in a neighborhood $N(x^*,\varepsilon)$ of $x^*$, and $B_k$ satisfies 
\begin{equation} \label{e.sup}
\lim_{k\rightarrow \infty} \frac{\|[B_k-H(x^*)]d_k\|}
{\|d_k\|}=0.
\end{equation}
then \\
(i) there exists a constant $k_1$ such that for all $k \geq k_1$ we have $x_{k+1} = x_k + d_k$;\\
(ii) the sequence $\{x_k\}$ generated by Algorithm 2 converges to $x^*$ superlinearly.
\end{thm}

%%%%%%%%%%%%%
\begin{proof}
(i) The condition (\ref{e.sup}) implies
\begin{equation}\label{e:39}
\lim_{k\rightarrow\infty}\frac{\|g_k+H(x^*)d_k\|}{\|d_k\|}=0,
\end{equation}
leading to
\begin{equation*}
d_k=-H(x^*)^{-1} g_k+o(\|d_k\|).
\end{equation*}
This implies that
\begin{equation} \label{e.dk}
\|d_k\|\leq \|H(x^*)^{-1}\|~ \|g_k\|+o(\|d_k\|).\end{equation}
Theorem \ref{t.sgolc} implies that $\|g_k\| \rightarrow 0$, as $k\rightarrow \infty$. This and (\ref{e.dk}) give
\begin{equation}\label{e:40}
\lim_{k\rightarrow \infty} \|d_k\|=0.
\end{equation}
This, (\ref{e.ine}), and (H2) imply
\begin{equation*}\begin{split}
|r_k-1|&=\left|\frac{f_k-f(x_k+d_k)}{q_k(0)-q_k(d_k)}-1\right|=\left|\frac{f_k-f(x_k+d_k)-(q_k(0)-q_k(d_k))}{q_k(0)-q_k(d_k)}\right|\\
&\leq \frac{O(\|d_k\|^2)}{\beta \varepsilon~ \min\left\{\delta_k,\varepsilon/M \right\}} \leq \frac{O(\|d_k\|^2)}{\beta \varepsilon~ \min\left\{\|d_k\|_k,\varepsilon/M \right\}} \rightarrow 0 ~~~ (k\rightarrow\infty).
\end{split}\end{equation*}
This clearly implies that there exists a positive integer $k_1$ such that for $k \geq k_1$ we have $x_{k+1} = x_k + d_k$.

(ii) From $d_k = - B_k^{-1} g_k$, we obtain
\[
\frac{\|g_k+H_k d_k\|}{\|d_k\|} = \frac{\|[H_k- B_k]d_k\|}{\|d_k\|} \leq \frac{\|[H_k-H(x^*)]d_k\|}{\|d_k\|} + \frac{\|[B_k-H(x^*)]d_k\|}{\|d_k\|}.
\]
This and (\ref{e:28}) lead to 
\begin{equation}\label{e:355558}
\lim_{k\rightarrow \infty} \frac{\|g_k+H_k d_k\|}{\|d_k\|}=0.
\end{equation}
Now Theorem 3.6 in \cite{NocW} implies that $\{x_k\}$ generated by Algorithm 2 converges to $x^*$ superlinearly. \qed
\end{proof}
%%%%%%%%%%%

Notice that if $f$ is thrice continuously differentiable 
and the upper level set $L(x_0)$ is bounded, then (H1) implies that $\|\nabla^3f(x)\|$ is uniformly continuous and bounded on the open bounded convex set $\Omega$ involving $L(x_0)$. Hence, by using the mean value theorem, there exists a constant $L>0$ such that $\|\nabla ^3f(x)\|\leq L$ implying
\begin{equation} \label{e.lip}
\|H(x)-H(y)\|\leq L\|x-y\|, 
\end{equation}
for all $x, y \in \Omega$. This implies that Hessian of $f$ is Lipschitz continuous. This condition can guarantee the quadratic convergence of the sequence $\{x_k\}$ generated by Algorithm 2. The details are summarized in the next result.

%%%%%%%%%%%%%%%%%%%%%%%%%%
\begin{thm} \label{t.quad}
Suppose that $f(x)$ is a twice continuously differentiable function on $\mathbb{R}^n$, and all assumptions of Theorem \ref{t.sgolc} hold. If $\|d_k\| = \|-H_k^{-1}g_k\| \leq \delta_k$, and there exists a neighborhood $N(x^*,\epsilon)$ of $x^*$ such that $H(x)$ is Lipschitz continuous on $N(x^*,\epsilon)$, i.e., there exists $L$ such that
\begin{equation}\label{e:42}
\|H(x)-H(y)\| \leq L \|x-y\|,
\end{equation}
then \\
(i) there exists a constant $k_2$ such that for all $k \geq k_2$ we have $x_{k+1} = x_k + d_k$;\\
(ii) the sequence $\{x_k\}$ generated by Algorithm 2 converges to $x^*$ quadratically.
\end{thm}

%%%%%%%%%%%%%
\begin{proof}
(i) By replacing $B_k$ by $H_k$ in Theorem \ref{t.sup}, we obtain that there exists an integer $k_2>0$ such that
\[
x_{k+1} = x_k - H_k^{-1}g_k,
\]
for all $k \geq k_1$. 

(ii) The condition described in (i) and Theorem 3.5 in \cite{NocW} give the results. \qed
\end{proof}
%%%%%%%%%%%

% ################################################
% ################################################
\section{Numerical experiments}
In this section we report numerical results for Algorithm 2 equipped with two novel nonmonotone terms proposed in Section 2 for solving unconstrained optimization problems. In our experiments we use several version of Algorithm 2 employing state-of-the-art nonmonotone terms. In details, we consider
{\renewcommand{\labelitemi}{$\bullet$} 
\begin{itemize}
\item NMTR-G: Algorithm 2 with the nonmonotone term of {\sc Grippo} et al. \cite{GriLL1};
\item NMTR-H: Algorithm 2 with the nonmonotone term of {\sc Zhang \& Hager} \cite{ZhaH};
\item NMTR-N: Algorithm 2 with the nonmonotone term of {\sc Amini} et al. \cite{AmiAN};
\item NMTR-M: Algorithm 2 with the nonmonotone term of {\sc Ahookhosh} et al. \cite{AhoAB};
\item NMTR-1: Algorithm 2 with the nonmonotone term (\ref{e.tk1});
\item NMTR-2: Algorithm 2 with the nonmonotone term (\ref{e.tk2}).
\end{itemize} 
In the experiments we used 112 test problems of the \textsf{CUTEst} test collections \cite{GolOT} from dimension 2 to 5000, where we ignore test problems with the dimension greater than 5000. All of the codes are written in MATLAB using the same subroutine, and they are tested on 2Hz core i5 processor laptop with 4GB of RAM with the double-precision data type. The initial points are standard ones proposed in \textsf{CUTEst}. All the algorithms use the radius
\begin{equation*}
\delta_{k+1}= \left\{
\begin{array}{ll}
c_1\|d_k\| & ~~\mathrm{if}~ \widehat{r}_k < \mu_1,\\
\delta_k& ~~\mathrm{if}~ \mu_1 \leq \widehat{r}_k < \mu_2, \\
\max \{\delta_k,c_2 \|d_k\| \} & ~~\mathrm{if}~ \widehat{r}_k \geq \mu_2, \\
\end{array}  \right. \\
\end{equation*}
where 
\begin{equation*}
\mu_{1}=0.05,~\mu_{2} = 0.9,~ c_1=0.25,~ c_2=2.5,~ \delta_{0} =0.1\Vert gk \Vert, 
\end{equation*}
see \cite{GolOST}. In the model $q_k$ (\ref{e.qk}), an approximation for Hessian is generated by the BFGS updating formula
\begin{equation*} B_{k+1}=B_k + \frac{y_k y_k^T}{s_k^T y_k}-\frac{B_k s_k s_k^T B_k}{s_k^T B_k s_k}, \end{equation*}
where $s_k=x_{k+1}-x_k$ and $y_k=g_{k+1}-g_k$.
For NMTR-G, NMTR-N, NMTR-1 and NMTR-2, we set $N = 10$. As discussed in \cite{ZhaH}, NMLS-H uses $\eta_k = 0.85$.  On the basis of our experiments, we update the parameter $\eta_k$ by 
\begin{equation*}\label{e.eta} 
\eta_{k}=\left\{ 
\begin{array}{ll} 
\eta_0/2   \vspace{4mm}    &  ~~ \mathrm{if}~ k=1, \\ 
(\eta_{k-1}+\eta_{k-2})/2  &  ~~ \mathrm{if}~ k\geq 2,
\end{array} \right.
\end{equation*} 
for NMTR-N, NMTR-M, NMTR-1 and NMTR-2, where the parameter $\eta_0$ will be tuned to get a better performance. To solve the quadratic subproblem (\ref{e.sub}), we use the Steihaug-Toint scheme \cite{ConGT} (Chapter 7, Page 205) where the scheme is terminated if
\begin{equation*}
\|g(x_k+d)\| \leq \min\left\{1/10, \|g_k\|^{1/2}\right\}\|g_k\|~~~ \textrm{or}~~~
\|d\|=\delta_k.
\end{equation*}
In our experiments the algorithms are stopped whenever the total number of iterations exceeds 10000 or     
\begin{equation} 
\|g_k\| < \varepsilon 
\end{equation} 
holds with the accuracy parameter $\varepsilon = 10^{-5}$. 

To compare the results appropriately, we use the performance profiles of {\sc Dolan \& Mor\'{e}} in \cite{DolM}, where the measures of performance are the number of iterations ($N_i$), the number of function evaluations ($N_f$), and the number of gradient evaluations ($N_g$). In the algorithms considered, the number of iterations and gradient evaluations are the same, so we only consider the performance of gradients. It is believed that the computational cost of a gradient is as much as the computational cost three function values, i.e., we further consider the measure $N_f+3 N_g$. The performance of each code is measured by considering the ratio of its computational outcome versus the best numerical outcome of all codes. This profile offers a tool for comparing the performance of iterative schemes in a statistical structure. Let $\mathcal{S}$ be a set of all algorithms and $\mathcal{P}$ be a set of test problems. For each problem $p$ and solver $s$,  $t_{p,s}$ is the computational outcome regarding to the performance index, which is used in defining the performance ratio 
\begin{equation}\label{e.perf}
r_{p,s}=\frac{t_{p,s}}{\min \{ t_{p,s}: s\in \mathcal{S}\}}.
\end{equation} 
If an algorithm $s$ is failed to solve a problem $p$, the procedure sets $r_{p,s}=r_{\mathrm{failed}}$, where $r_{\mathrm{failed}}$ should be strictly larger than any performance ratio (\ref{e.perf}). For any factor $\tau$, the overall performance of an algorithm $s$ is given by 
\begin{equation*} 
\rho_s(\tau)=\frac{1}{n_p} \textrm{size}\{p \in \mathcal{P}: r_{p,s} \leq \tau\}.  
\end{equation*} 
In fact $\rho_s(\tau)$ is the probability that a performance ratio $r_{p,s}$ of the algorithm $s\in \mathcal{S}$ is within a factor $\tau \in \mathbb{R}^n$ of the best possible ratio. The function 
$\rho_s(\tau)$ is a distribution function for the performance ratio. In particular, $\rho_s(1)$ gives the probability that the algorithm $s$ wins over all other considered algorithms, and $\lim_{\tau\rightarrow 
r_{\mathrm{failed}}}\rho_s(\tau)$ gives the probability of that the algorithm $s$ 
solve all considered problems. Hence the performance profile can be considered as a measure of efficiency for comparing iterative schemes. In Figures \ref{fpar} and \ref{fnon}, the x-axis shows the number $\tau$ while 
the y-axis inhibits $P(r_{p,s} \leq \tau: 1 \leq s \leq n_s)$.

% ################################################
\subsection{{\bf Experiments with highly nonlinear problems}}
In this subsection we give some numerical results regarding 
the implementation of NMTR-1 and NMTR-2 compared with
TTR on some two-dimensional 
highly nonlinear problems involving a curved narrow valley.
More precisely, we consider  
the Nesterov-Chebysheve-Rosenbrock, Maratos, and NONDIA 
functions, see, for example, \cite{And}. In Example 1 the 
Nesterov-Chebysheve-Rosenbrock function is given, and 
the Maratos and NONDIA functions are given by
\[
\begin{array}{l}
\D f(x_1,x_2) =x_1 + \theta_1 ( x_1^2 + x_2^2 - 1 )^ 2
~~~ (\mathrm{Maratos~ function})\\\\
\end{array}
\]
and
\[
\D f(x_1,x_2) = (1-x_2)^2 + \theta_2( x_1 - x_2^2)^2
~~~(\mathrm{NONDIA~ function}),
\]
respectively, where we consider $\theta_1 = 10$ and $\theta_2 = 100$.

We solve the problem (\ref{e.func}) for these three
functions using TTR, NMTR-1, and NMTR-2, and the
results regarding the number of iterations and 
function evaluations are summarized in Table \ref{table1}. 
To give a clear view of the behaviour of TTR, NMTR-1, 
and NMTR-2, we depict the contour plot of the considered
functions and iterations obtained by the algorithms in Figure 
\ref{fcon2} (a), (c), and (e). In all three cases, one
can see that NMTR-1 and NMTR-2 need less iterations and
function values compared with TTR to solve the problem.
Moreover, TTR behaves monotonically and follows the bottom 
of the associated valley, while NMTR-1 and NMTR-2 fluctuated in the
valley.  

%%%%%%%%%%%%%%%%%%%%%%%%%%
\begin{figure} 
 \centering 
 \subfloat[][Nes-Cheb-Rosen contour plot \& iterations]
 {\includegraphics[width=7.7cm]{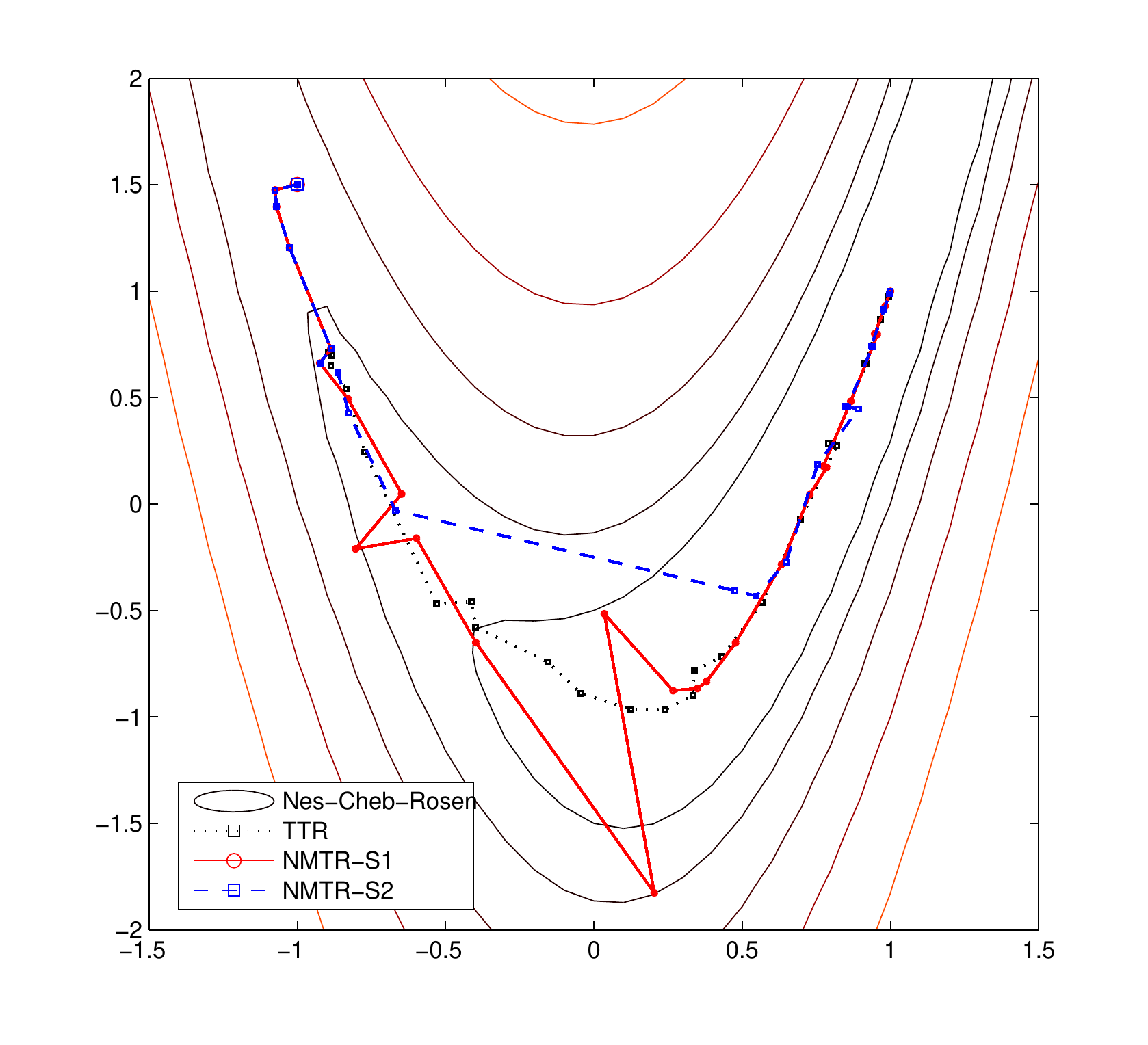}}% 
 \qquad 
 \subfloat[][function values versus iterations]
 {\includegraphics[width=7.7cm]{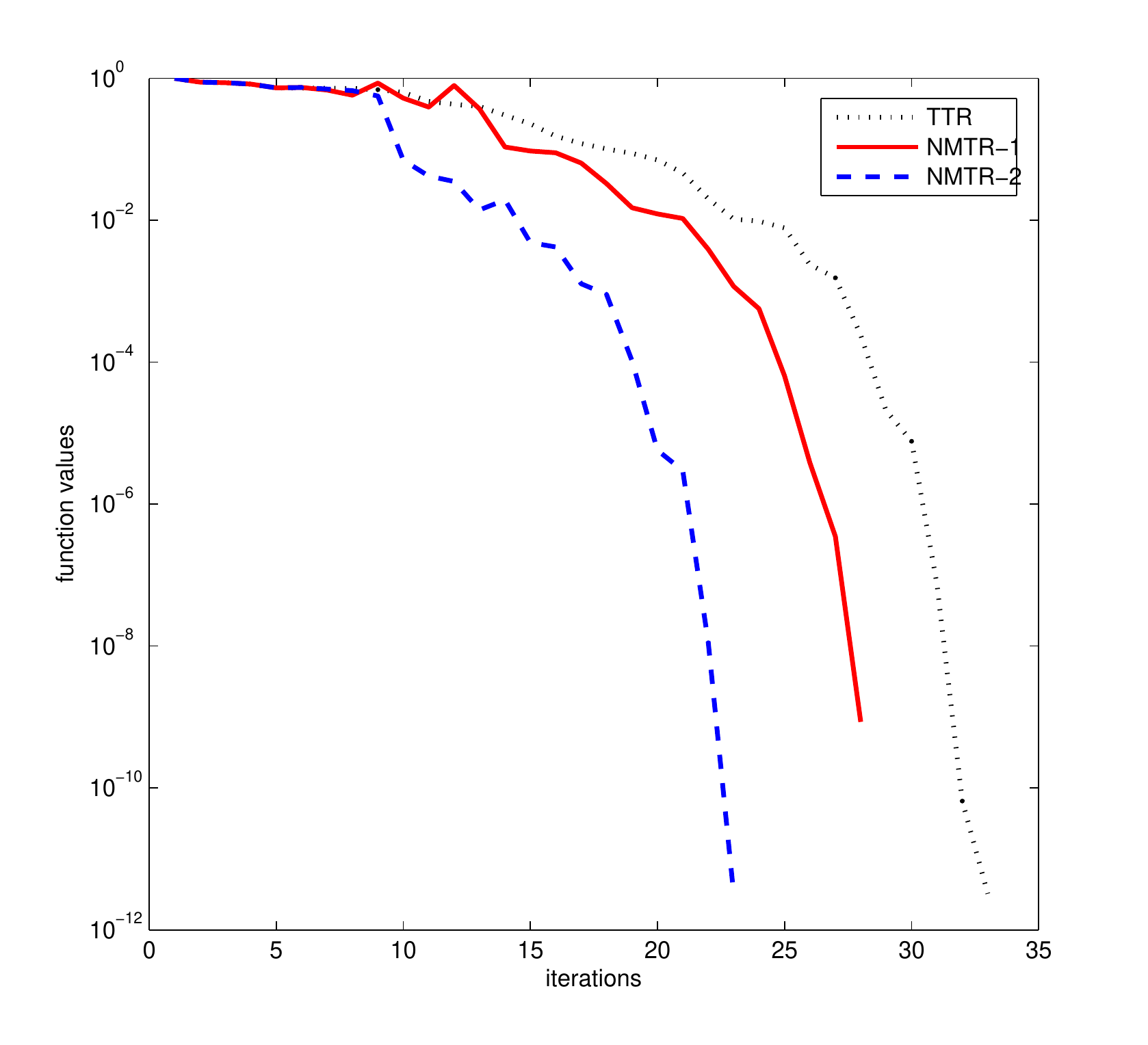}} 
 \qquad 
 \subfloat[][Maratos contour plot \& iterations]
 {\includegraphics[width=7.7cm]{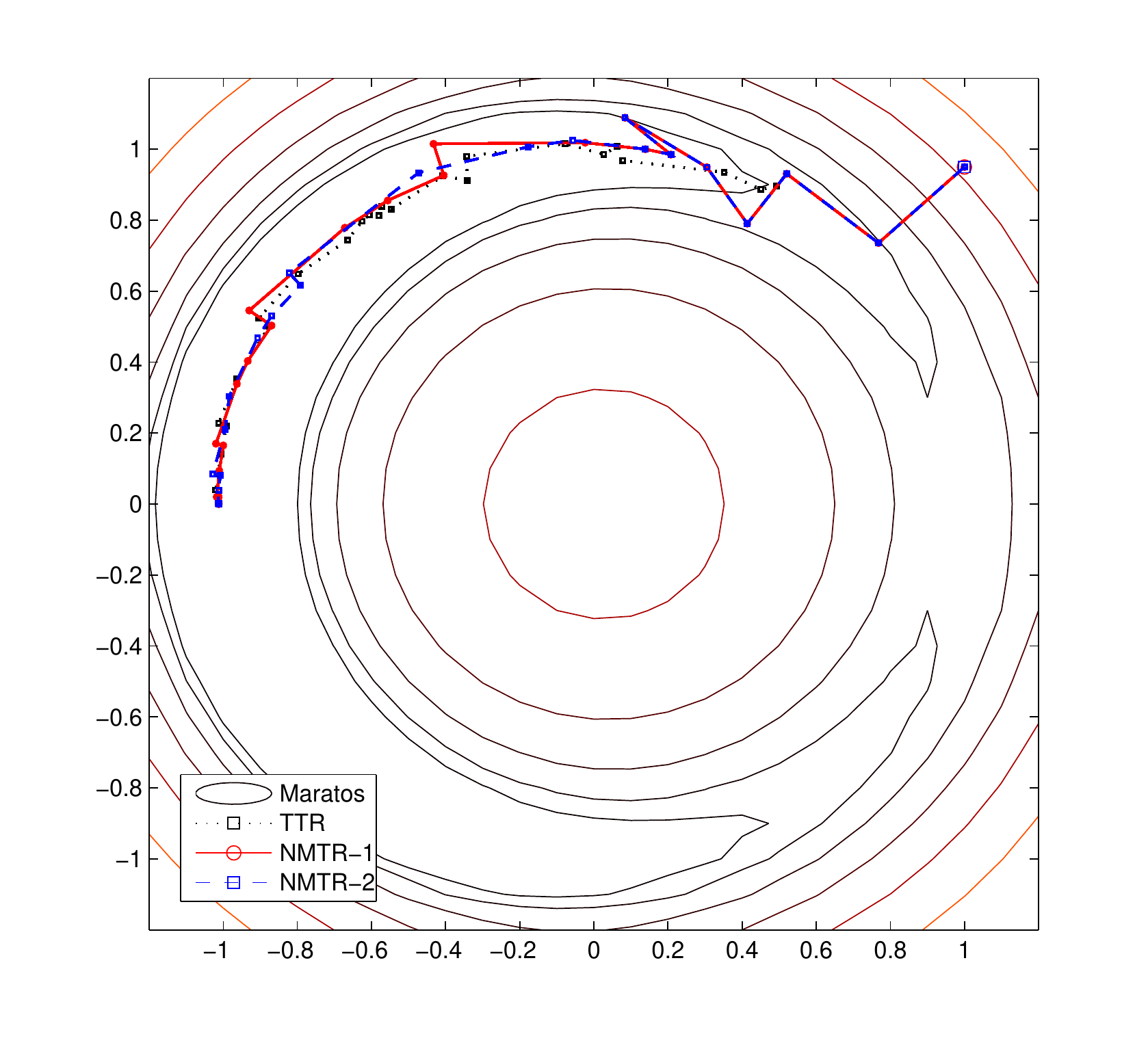}}% 
 \qquad 
 \subfloat[][function values versus iterations]
 {\includegraphics[width=7.7cm]{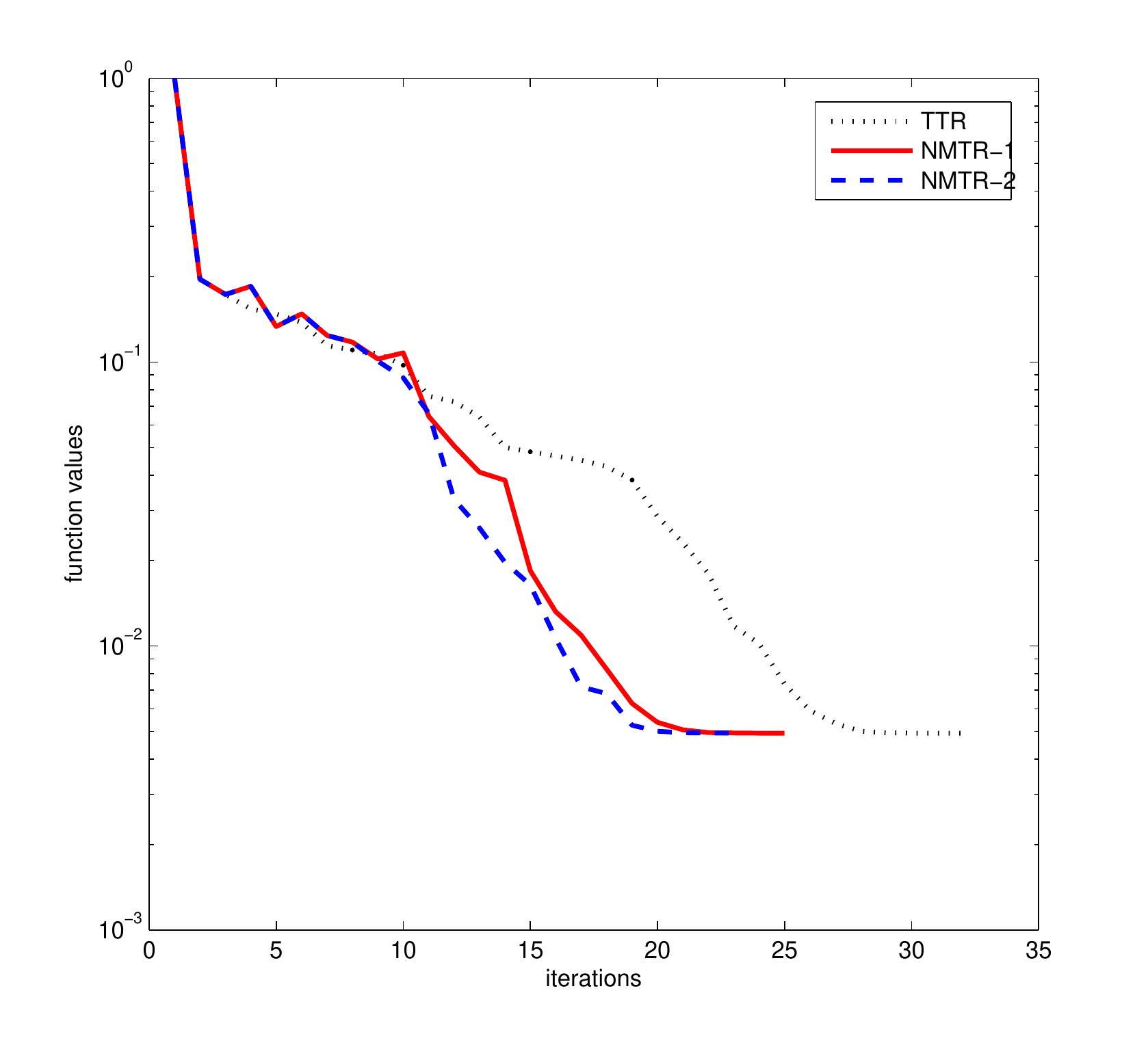}} 
 \qquad 
 \subfloat[][NONDIA contour plot \& iterations]
 {\includegraphics[width=7.7cm]{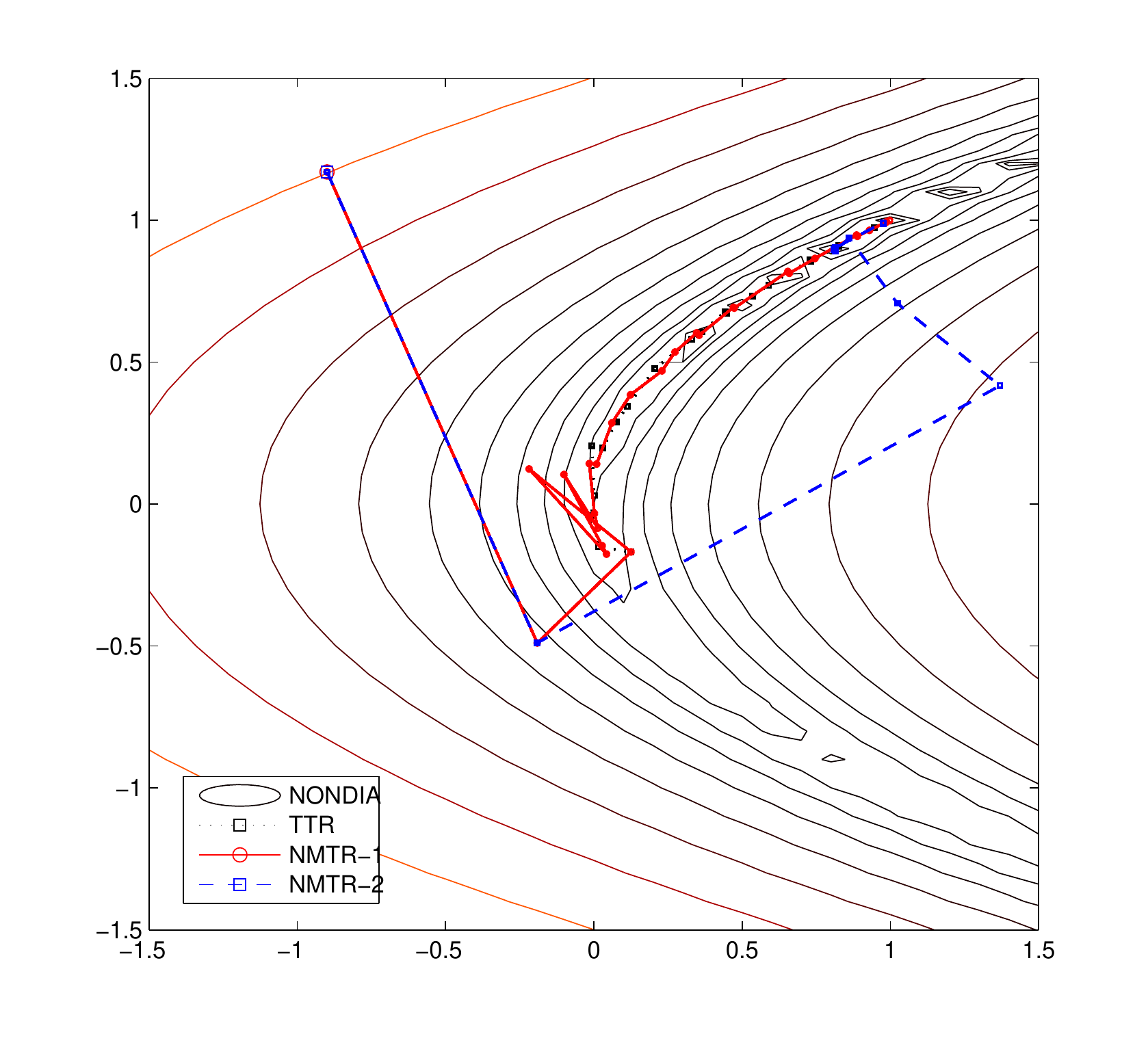}}% 
 \qquad 
 \subfloat[][function values versus iterations]
 {\includegraphics[width=7.7cm]{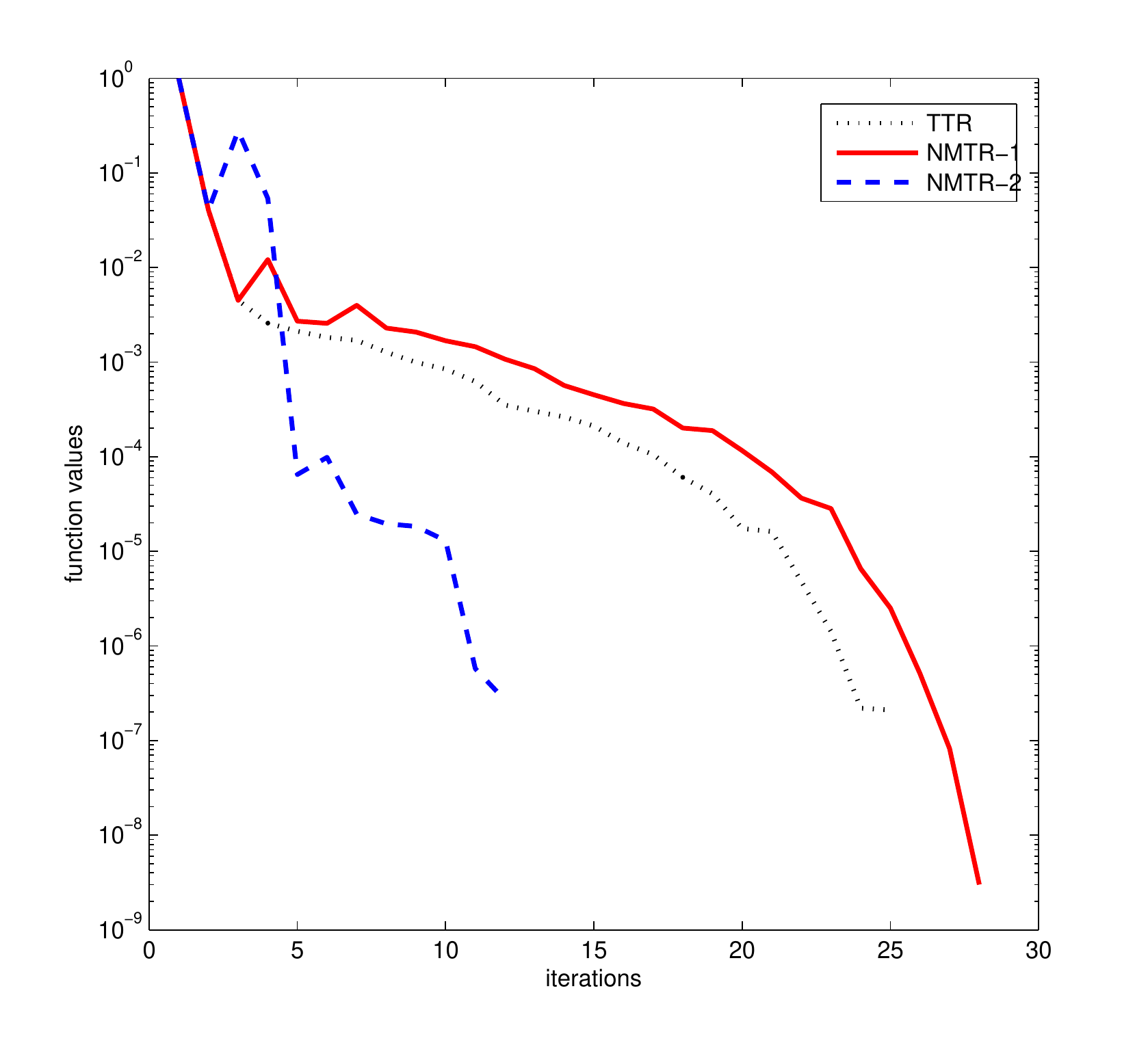}}% 
 \qquad 
 \caption{A comparison among NMTR-1, MNTR-2, and TTR: Subfigures (a), (c), 
 and (e) respectively illustrate the contour plots of the 
 two-dimensional Nesterov-Chebysheve-Rosenbrock, Maratos, 
 and NONDIA functions and iterations of NMTR-1, MNTR-2, 
 and TTR; Subfigures (b), (d), and (f) show the diagram of function 
 values versus iterations.} 
\label{fcon2}
\end{figure}
%%%%%%%%%%%%

%%%%%%%%%%%%%%%%%%%%%%%%%%%%%%%%%%
\renewcommand{\arraystretch}{1.2}  
\begin{longtable}{lllllllllll} 
\multicolumn{5}{l} 
{Table 1. Numerical results for highly nonlinear problems}
\label{table1} \\ 
\cmidrule{1-9} 
\multicolumn{1}{l}{Problem name} &\multicolumn{1}{l}{Dim}  
&\multicolumn{1}{l}{Initial point}    
&\multicolumn{1}{l}{TTR}    &\multicolumn{1}{l}{} 
&\multicolumn{1}{l}{NMTR-1} &\multicolumn{1}{l}{} 
&\multicolumn{1}{l}{NMTR-2} &\multicolumn{1}{l}{} \\ 
\cmidrule(lr){4-5} \cmidrule(lr){6-7} \cmidrule(lr){8-9}  
\multicolumn{1}{l}{}        &\multicolumn{1}{l}{} 
&\multicolumn{1}{l}{}  
&\multicolumn{1}{l}{$N_g$}   &\multicolumn{1}{l}{$N_f$} %&\multicolumn{1}{l}{$\mathrm{T}$}
&\multicolumn{1}{l}{$N_g$}   &\multicolumn{1}{l}{$N_f$}
&\multicolumn{1}{l}{$N_g$}   &\multicolumn{1}{l}{$N_f$} 
\\ %&\multicolumn{1}{l}{$\mathrm{T}$} \\ 
\cmidrule{1-9} 
\endfirsthead 
\multicolumn{5}{l}% 
{Table 1. Numerical results (\textit{continued})}\\[5pt] 
\hline 
\endhead 
\hline 
\endfoot 
\endlastfoot 
Nes-Cheb-Rosen &2&(-1, 1.5)   &32&41&27&34&22&29\\ 
Maratos        &2&(1, 0.95)   &31&40&24&29&22&29\\ 
NONDIA         &2&(-0.9, 1.17)&24&34&27&34&11&17\\
\hline
\end{longtable}  
%%%%%%%%%%%%%%%

% ################################################
\subsection{{\bf Experiments with \textsf{CUTEst} test problems}}
In this subsection we give numerical results regarding experiments 
with NMTR-1 and NMTR-2 on the \textsf{CUTEst} test problems 
compared with NMTR-G, NMTR-H, NMTR-N, and NMTR-M.

To get a better performance from NMTR-1 and NMTR-2, we tune the parameter $\eta_0$ by testing several fixed values of $\eta_0$ for both algorithms, where we use  $\eta_0=0.15, 0.25, 0.35, 0.45$. The corresponding versions of the algorithms NMTR-1 and NMTR-2 are denoted by NMTR-1-0.15, NMTR-1-0.25, NMTR-1-0.35, NMTR-1-0.45, NMTR-2-0.15, NMTR-2-0.25, NMTR-2-0.35, and NMTR-2-0.45, respectively. The results are summarized in Figure \ref{fpar} for three measures: the number function evaluations; the number gradient evaluations; the mixed measure $N_f + 3 N_g$. In Figure \ref{fpar}, subfigures (a), (c) and (e) illustrate that the results of NMTR-1, where it produces the best results with $\eta_{0} = 0.25$. From subfigures (b), (d), and (f) of Figure \ref{fpar}, it can be seen that the best results are produced by $\eta_{0}=0.45$. Hence for NMTR-1 we use $\eta_{0} = 0.25$ and for NMTR-2 use $\eta_{0} = 0.45$ in the remainder of our experiments.

%%%%%%%%%%%%%%%%%%%%%%%%%%
\begin{figure}\label{com1} 
 \centering 
 \subfloat[][$N_i$ and $N_g$ performance profile (NMTR-1)]{\includegraphics[width=7.7cm]{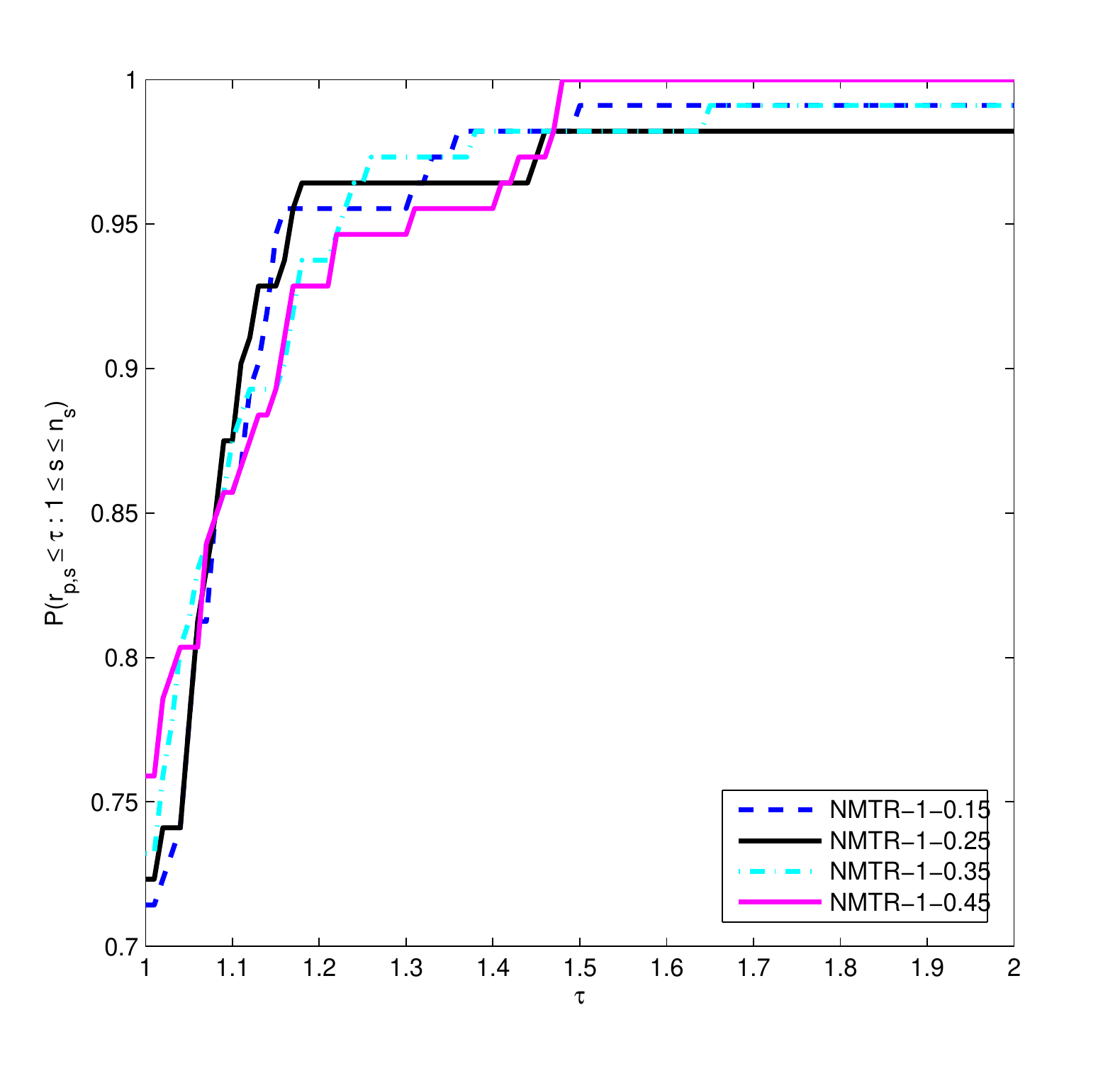}}% 
 \qquad 
 \subfloat[][$N_i$ and $N_g$ performance profile (NMTR-2)]{\includegraphics[width=7.7cm]{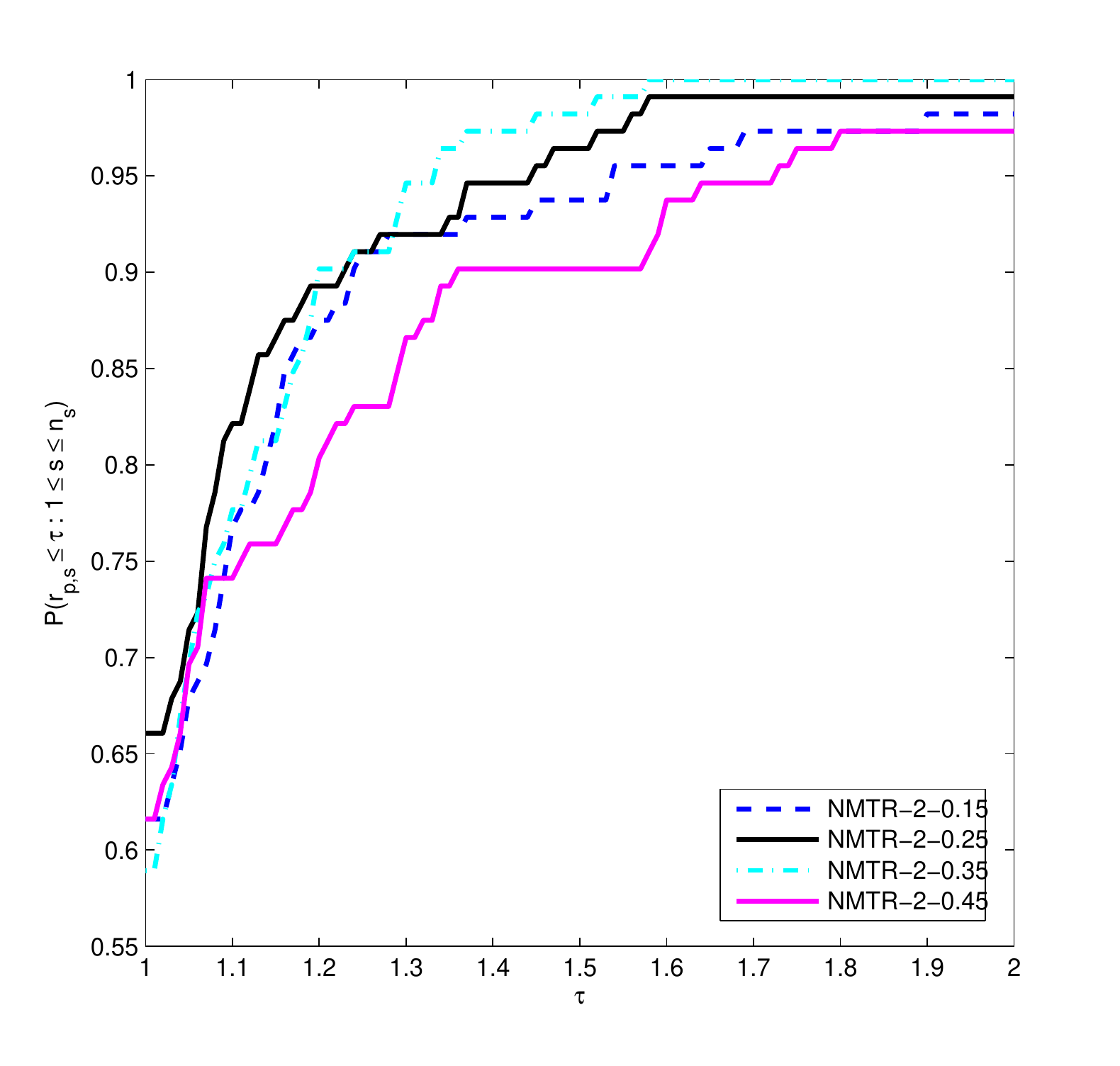}} 
 \qquad 
 \subfloat[][$N_f$ performance profile (NMTR-1)]{\includegraphics[width=7.7cm]{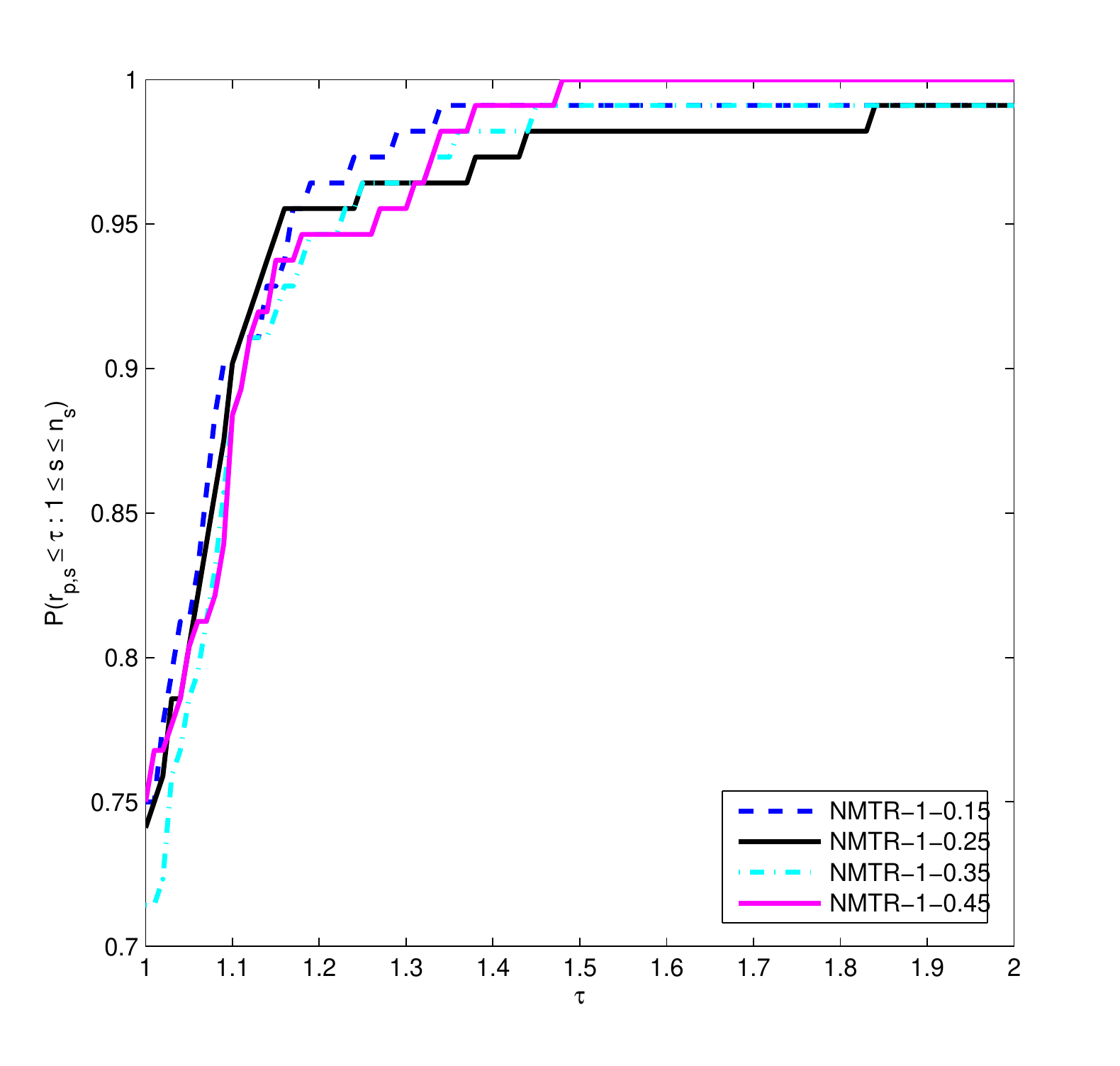}}% 
 \qquad 
 \subfloat[][$N_f$ performance profile (NMTR-2)]{\includegraphics[width=7.7cm]{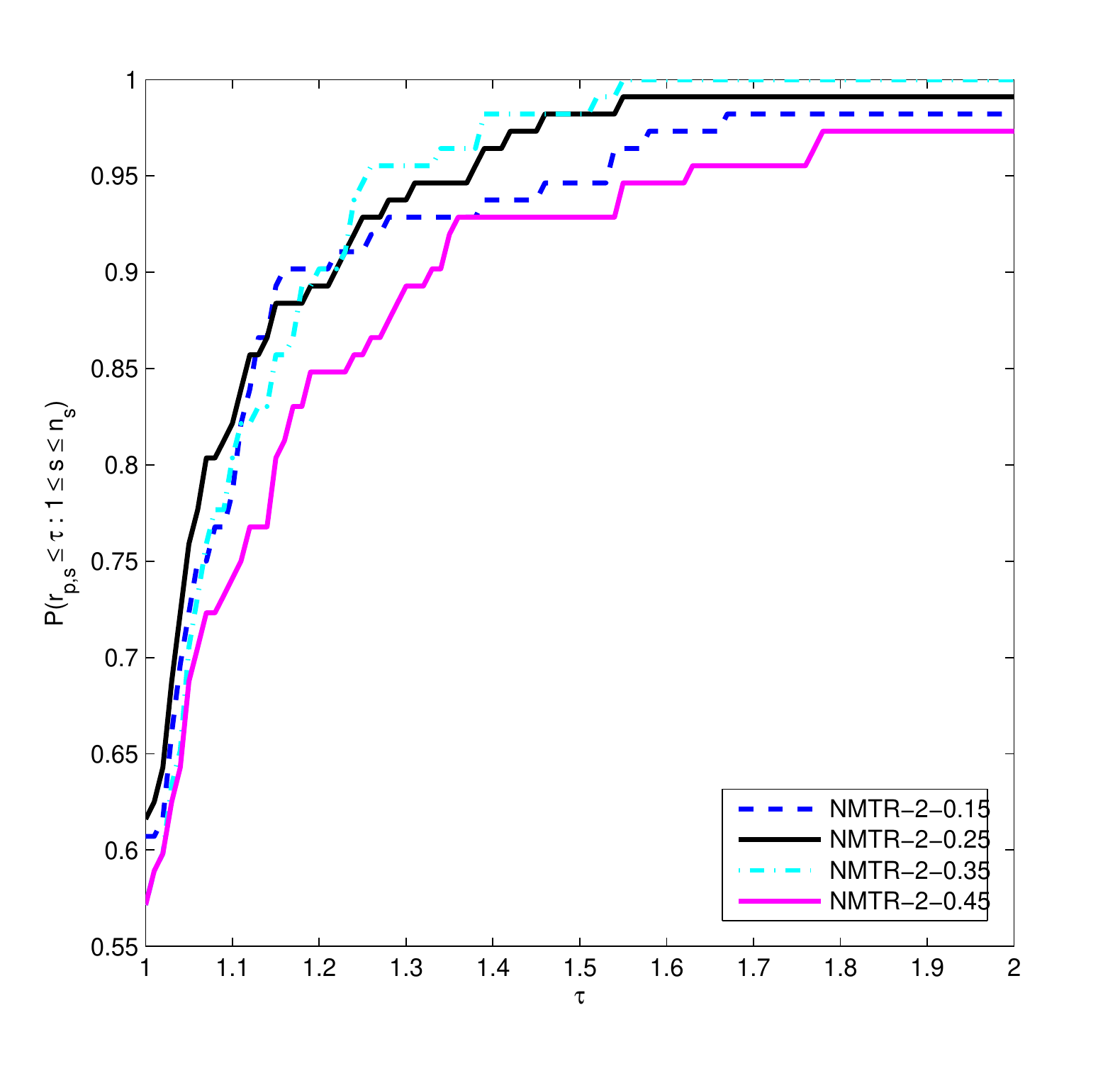}} 
 \qquad 
 \subfloat[][$N_f+3 N_g$ performance profile (NMTR-1)]{\includegraphics[width=7.7cm]{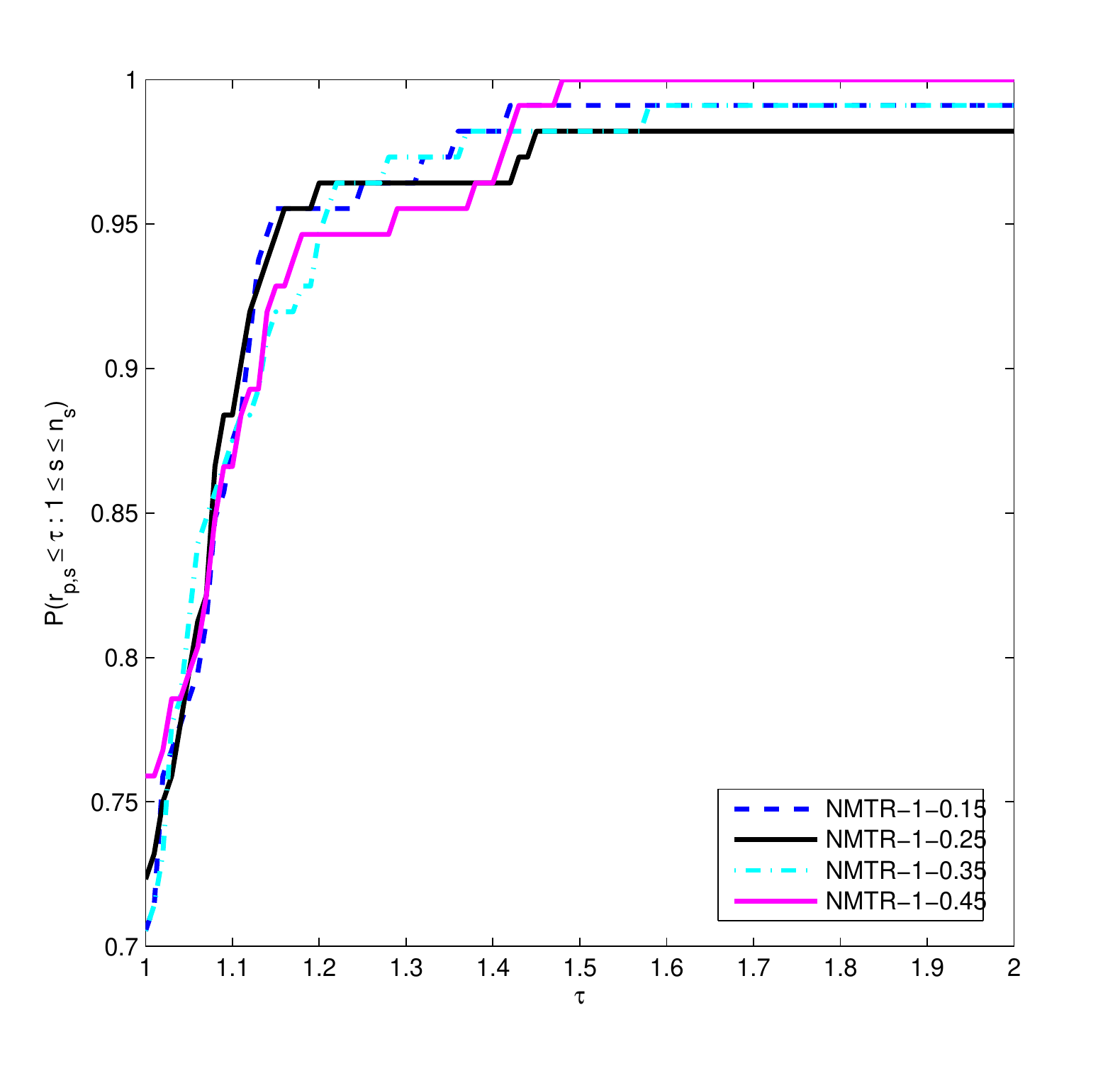}}% 
  \qquad 
  \subfloat[][$N_f+3 N_g$ performance profile (NMTR-2)]{\includegraphics[width=7.7cm]{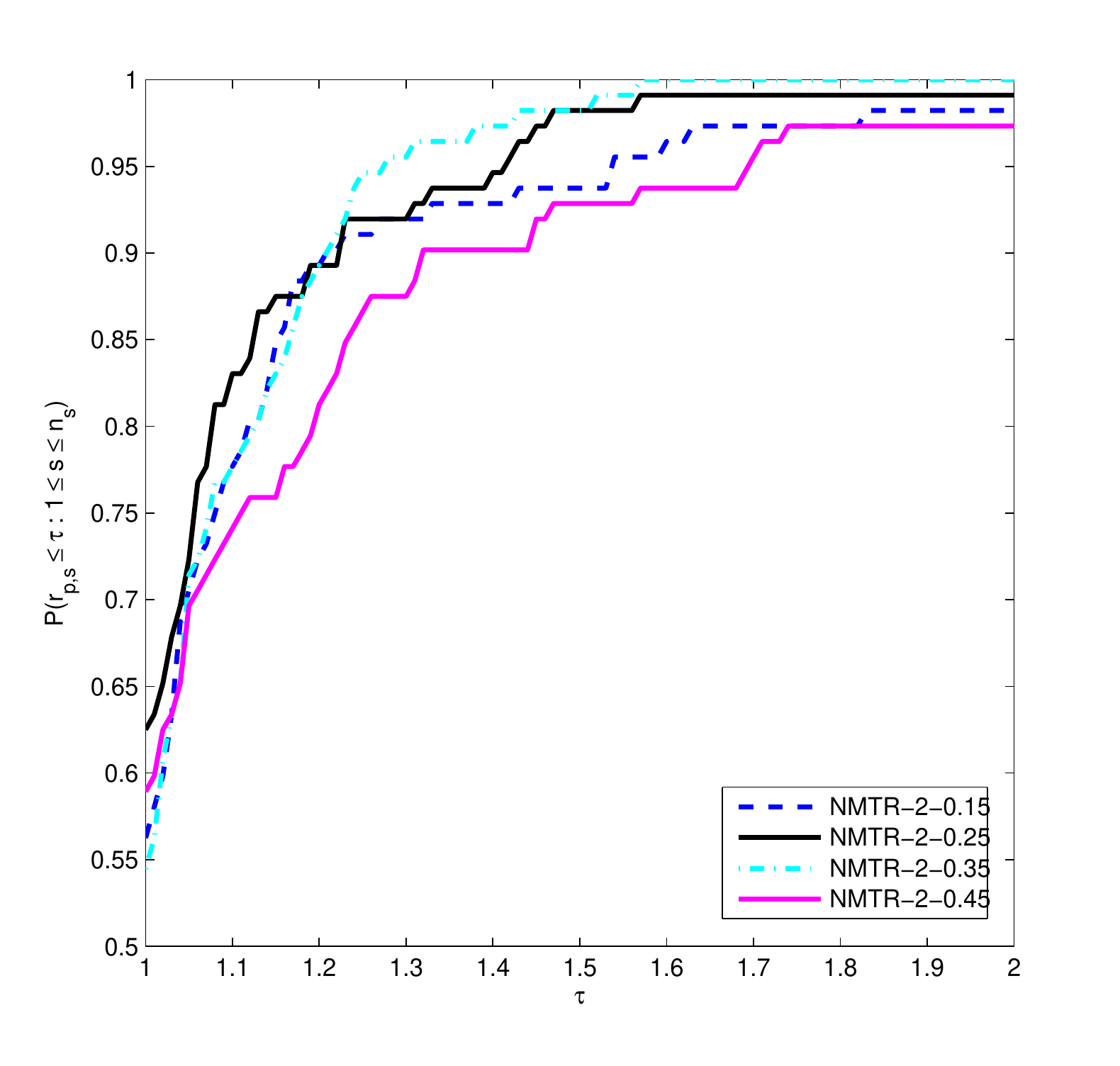}}% 
  \qquad 
 \caption{Performance profiles of NMTR-1 and NMTR-2 with the performance measures $N_g$, $N_f$, and $Nf+ 3 N_g$: Subfigures (a) and (b) display the number of iterations ($N_i$) or gradient evaluations ($N_g$); Subfigures (c) and (d) show the number of function evaluations ($N_f$); Subfigures (e) and (f) display the hybrid measure $N_f+3 N_g$.} 
\label{fpar} 
\end{figure}
%%%%%%%%%%%%

We here test NMTR-G, NMTR-H, NMTR-N, NMTR-M, NMTR-1, and NMTR-2 for solving the unconstrained problem (\ref{e.func}) and compare the produced results. The results of our implementations are summarized in Table \ref{table2}, where $N_g$ and $N_f$ are reported. The results of Table \ref{table2} show that NMTR-1 has a competitive performance compared with NMTR-G, NMTR-H, NMTR-N, NMTR-M, however, NMTR-2 produces the best results. To have a better comparison among these algorithms, we illustrate the results in Figure \ref{fnon} by performance profiles for the measures $N_g$, $N_f$, and $N_f+3 N_g$. 

In Figure \ref{fnon}, Subfigure (a) displays for the number of gradient evaluations, where the best results attained by NMTR-2 and then by NMTR-N with about 63\% and 52\% of the most wins, respectively. NMTR-1 is comparable with NMTR-G, NMTR-H, NMTR-N, but its diagram grows up faster than the others, which means its performance is close to the performance of the best method NMTR-2. Subfigure (b) shows for the number of function evaluations and has a similar interpretation of Subfigure (a), however, NMTR-2 attains about 60\% of the most wins. In Figure \ref{fnon}, Subfigures (c) and (d) display for the mixed measure $N_f+3 N_g$ with $\tau = 1.5$ and $\tau = 5.5$, respectively. In this case NMTR-2 outperforms the others by attaining about 58\% of the most wins, and the others have comparable results, however, the diagrams of NMTR-1 and NMTR-M grow up faster than the others implying that they perform close to the best algorithm NMTR-2. 

%%%%%%%%%%%%%%%%%%%%%%%%%%
\begin{figure} 
\centering 
\subfloat[][$N_i$ and $N_g$ performance profile]{\includegraphics[width=7.7cm]{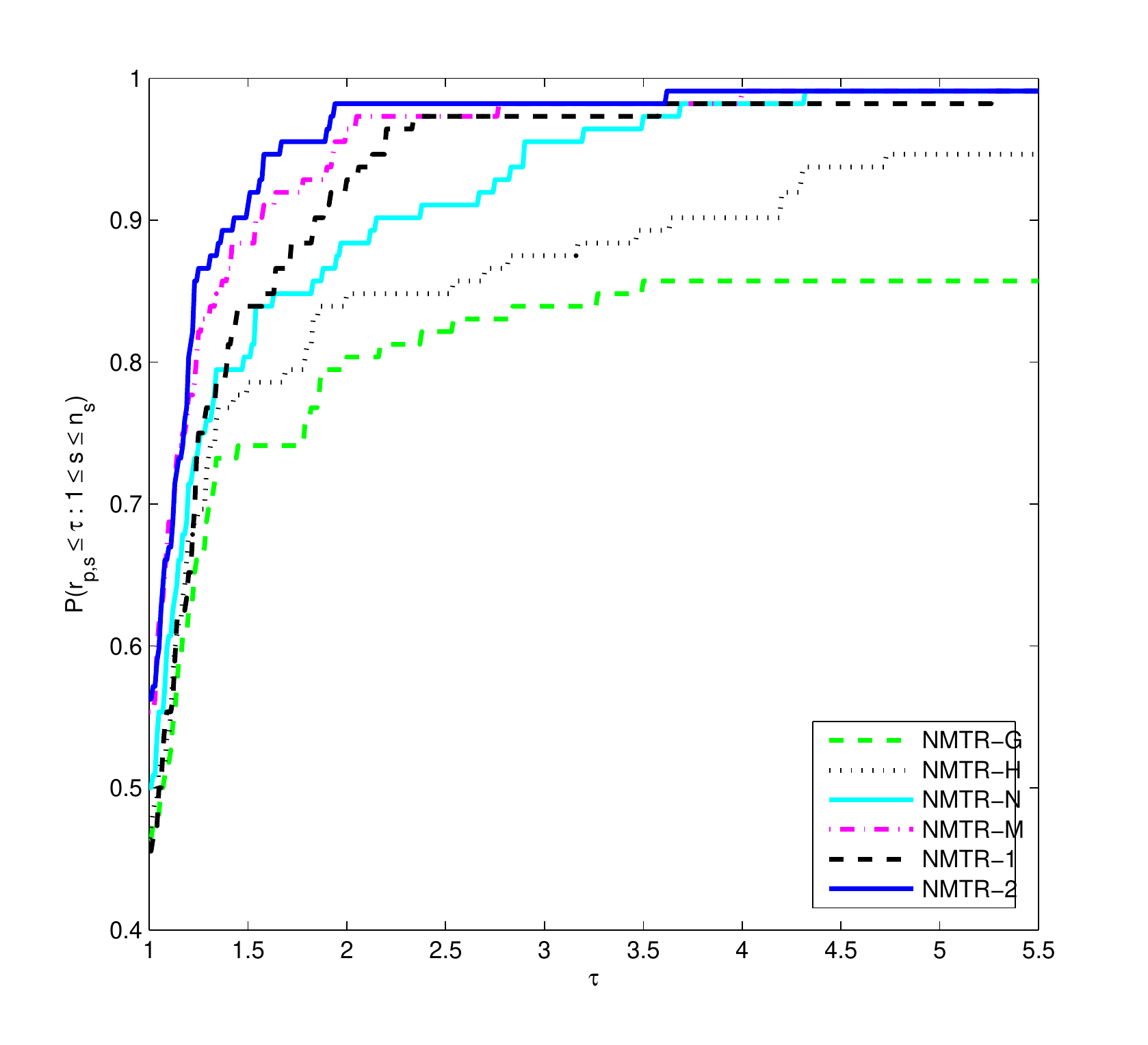}}% 
\qquad 
\subfloat[][$N_f$ performance profile]{\includegraphics[width=7.7cm]{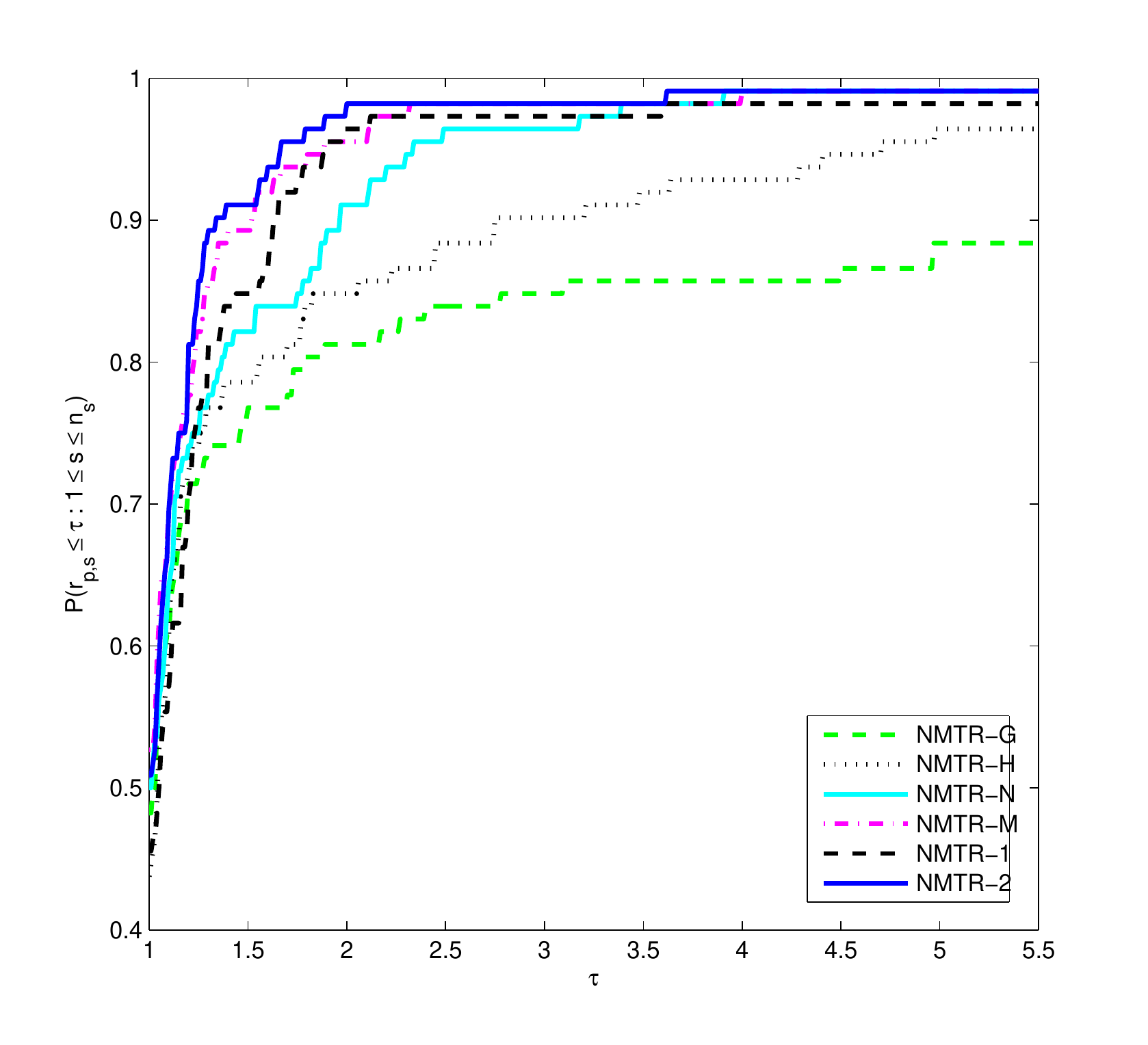}} 
\qquad 
\subfloat[][$N_f  +  3 N_g$ performance profile ($\tau = 1.5$)]{\includegraphics[width=7.7cm]{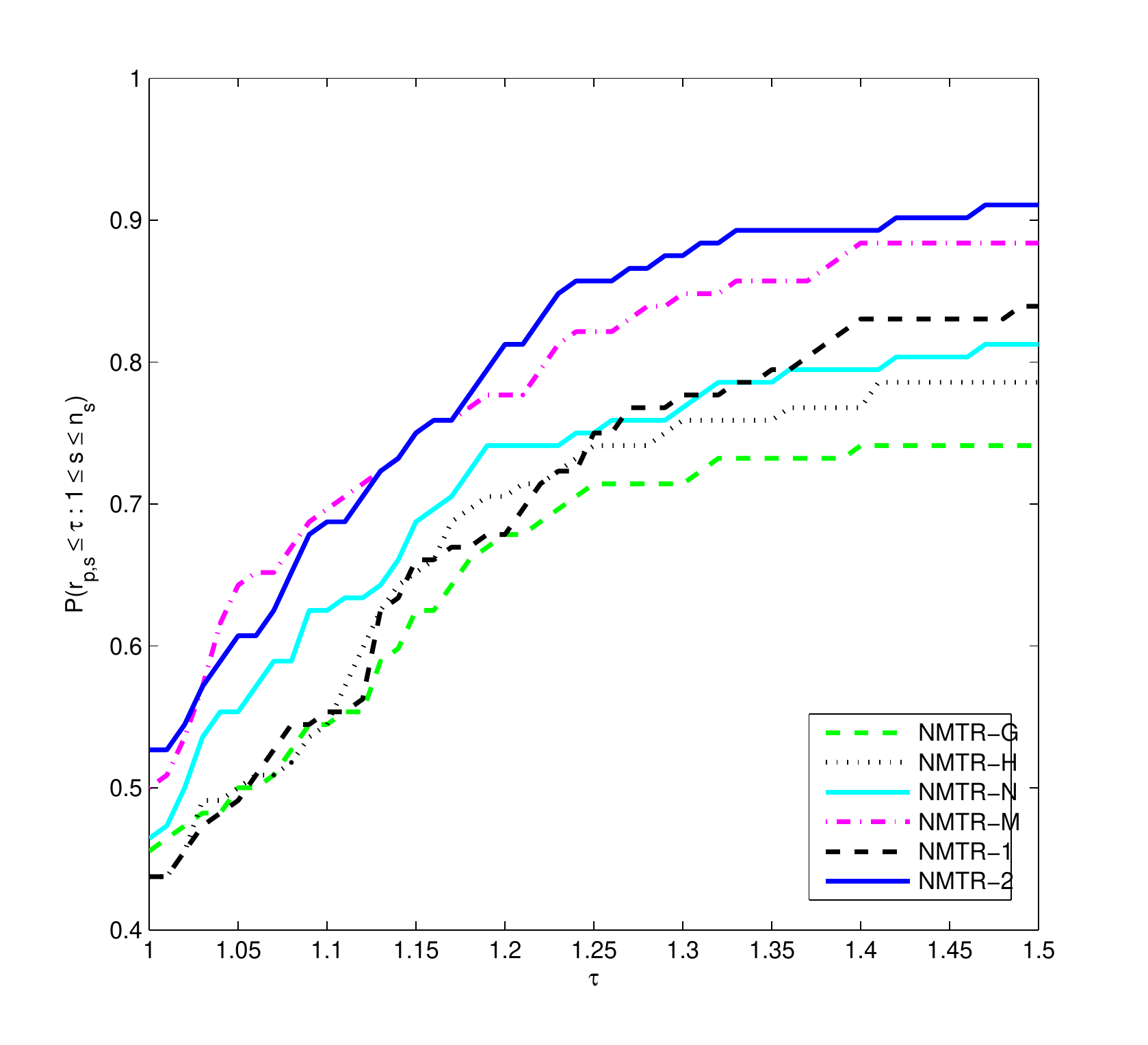}}% 
\qquad 
\subfloat[][$N_f  +  3 N_g$ performance profile ($\tau = 5.5$)]{\includegraphics[width=7.7cm]{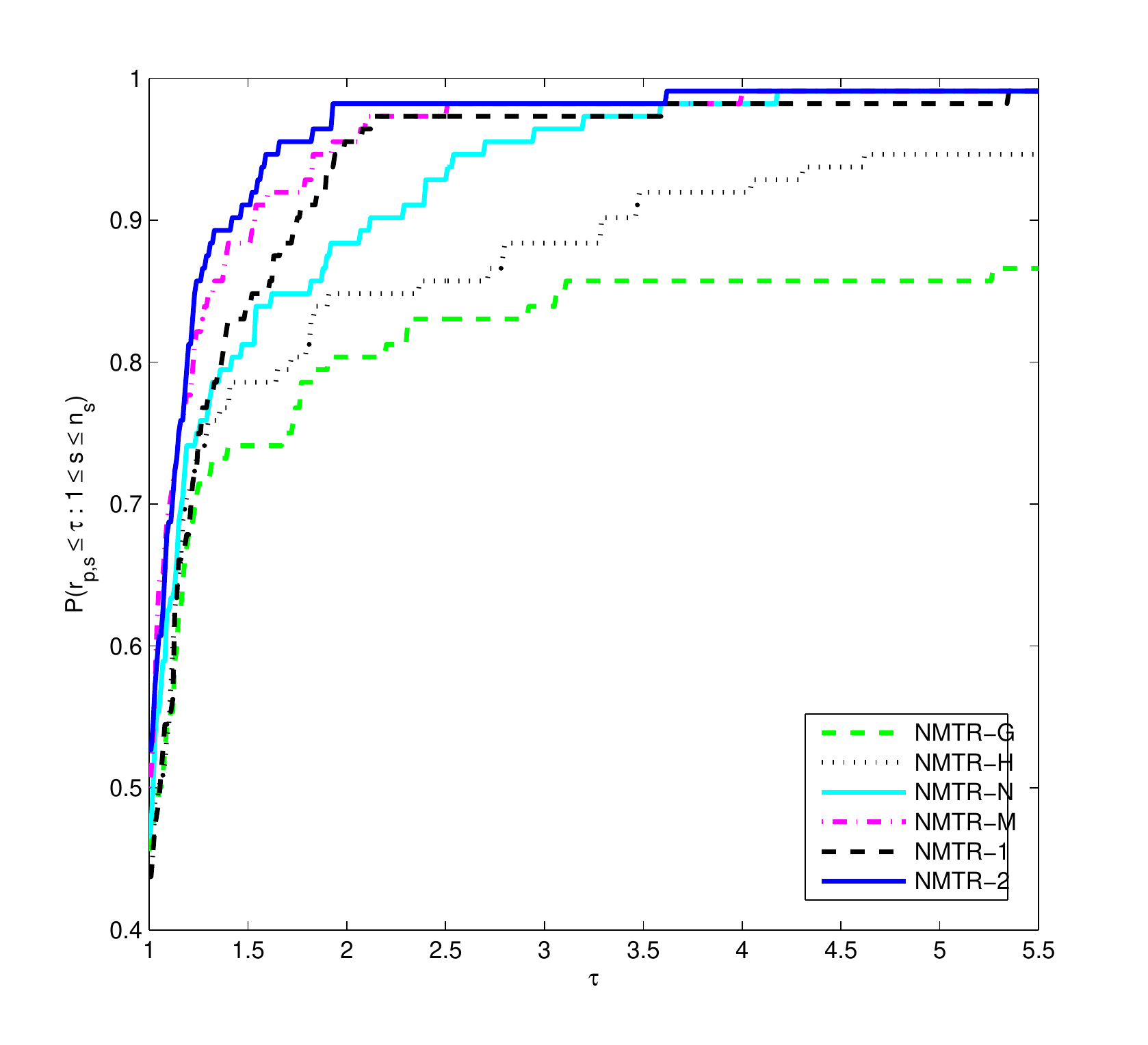}}%   
\caption{A comparison among NMTR-G, NMTR-H, NMTR-N, NMTR-M, NMTR-1, and NMTR-2 by performance profiles using the measures $N_g$, $N_f$, and $N_f+3 N_g$: Subfigure (a) displays the number of iterations ($N_i$) or gradient evaluations ($N_g$); Subfigure (b) shows the number of function evaluations ($N_f$); Subfigures (c) and (d) display the hybrid measure $N_f+3 N_g$ with $\tau = 1.5$ and $\tau = 5.5$, respectively.} 
\label{fnon} 
\end{figure} 
%%%%%%%%%%%%

% ################################################
% ################################################
\section {Concluding remarks}
In this paper we give some motivation for employing nonmonotone strategies in trust-region frameworks. Then we introduce two new nonmonotone terms and combine them into the traditional trust-region framework. It is shown that the proposed methods are golbally convergent to first- and second-order stationary points. Moreover local superlinear and quadratic convergence are established. Applying these methods on some highly nonlinear test problems involving a curved narrow valley show that they have a promising behaviour compared with the monotone trust-region method. Numerical experiments on a set of test problems from the \textsf{CUTEst} test collection show the efficiency of the proposed nonmonotone methods. 

Further research can be done in several aspects. For example, by combining the proposed nonmonotone trust-region methods with various adaptive radius, more efficient trust-region schemes can be derived, see, for example, \cite{AhoA2, AmiA}. The combination of the proposed nonmonotone terms with several inexact line searches such as Armijo, Wolfe, and Goldstein is also interesting, see \cite{AmiA}. The extension of the proposed method for constrained nonlinear optimization could be interesting, especially for nonnegativity constraints and box constraints, see, for example, \cite{BirMR,BirMR2,BonPTZ,PanT,Toi2,UlbU}. It also could be interesting to employ nonmonotone schemes for solving nonlinear least squares and system of nonlinear equations, see \cite{AhoEK} and references therein. Moreover, investigating new adaptive formulas for the parameter $\eta_k$ can be precious to improve the computational efficiency. \\\\
{\bf Appendix.} Table \ref{table2}.

%%%%%%%%%%%%%%%%%%%%%%%%%%%%%%%%%%
\renewcommand{\arraystretch}{1.18} 
\begin{landscape} 
\small 
\begin{longtable}{llllllllllllllll} 
\multicolumn{5}{l} 
{Table 2. Numerical results for nonmonotone trust-region methods}
\label{table2}\\ 
\cmidrule{1-14} 
\multicolumn{1}{l}{Problem name} &\multicolumn{1}{l}{Dimension}
&\multicolumn{1}{l}{NMTR-G}  &\multicolumn{1}{l}{}\hspace{10mm} 
&\multicolumn{1}{l}{NMTR-H} &\multicolumn{1}{l}{}\hspace{10mm} 
&\multicolumn{1}{l}{NMTR-N}  &\multicolumn{1}{l}{} \hspace{10mm} 
&\multicolumn{1}{l}{NMTR-M}  &\multicolumn{1}{l}{}\hspace{10mm} 
&\multicolumn{1}{l}{NMTR-1} &\multicolumn{1}{l}{}\hspace{10mm} 
&\multicolumn{1}{l}{NMTR-2}   &\multicolumn{1}{l}{} \\ 
\cmidrule(lr){3-4} \cmidrule(lr){5-6} \cmidrule(lr){7-8} 
\cmidrule(lr){9-10} \cmidrule(lr){11-12} \cmidrule(lr){13-14} 
\multicolumn{1}{l}{}        &\multicolumn{1}{l}{}
&\multicolumn{1}{l}{$N_g$}   &\multicolumn{1}{l}{$N_f$} %&\multicolumn{1}{l}{$\mathrm{T}$}
&\multicolumn{1}{l}{$N_g$}   &\multicolumn{1}{l}{$N_f$}
&\multicolumn{1}{l}{$N_g$}   &\multicolumn{1}{l}{$N_f$} 
&\multicolumn{1}{l}{$N_g$}   &\multicolumn{1}{l}{$N_f$} %&\multicolumn{1}{l}{$\mathrm{T}$} 
&\multicolumn{1}{l}{$N_g$}   &\multicolumn{1}{l}{$N_f$}% &\multicolumn{1}{l}{$\mathrm{T}$} 
&\multicolumn{1}{l}{$N_g$}   &\multicolumn{1}{l}{$N_f$}\\ %&\multicolumn{1}{l}{$\mathrm{T}$} \\ 
\cmidrule{1-14} 
\endfirsthead 
\multicolumn{5}{l}% 
{Table 2. Numerical results (\textit{continued})}\\[5pt] 
\hline 
\endhead 
\hline 
\endfoot 
\endlastfoot 
AIRCRFTB &8&27&39&26&39&24&34&22&37&26&38&22&37\\ 
ALLINITU &4&18&24&16&22&15&20&14&21&15&20&14&21\\ 
ARGLINA &200&3&4&3&4&3&4&3&4&3&4&3&4\\ 
ARGLINB &200&2&25&2&25&2&25&2&25&2&25&2&25\\ 
ARGLINC &200&2&25&2&25&2&25&2&25&2&25&2&25\\ 
ARWHEAD &5000&3&10&3&10&3&10&3&10&3&10&3&10\\ 
BARD &3&20&23&15&18&15&18&16&20&16&19&16&20\\ 
BDQRTIC &5000&153&223&73&109&33&56&17&30&21&38&17&30\\ 
BEALE &2&12&14&12&14&12&14&12&14&12&14&12&14\\ 
BIGGS3 &6&73&74&75&76&63&66&63&66&67&69&63&66\\ 
BIGGS5 &6&73&74&75&76&63&66&63&66&67&69&63&66\\ 
BIGGS6 &6&49&51&43&44&44&47&45&47&44&46&45&47\\ 
BOX2 &3&8&9&8&9&8&9&8&9&8&9&8&9\\ 
BOX3 &3&8&9&8&9&8&9&8&9&8&9&8&9\\ 
BRKMCC &2&6&8&6&8&6&8&6&8&6&8&6&8\\ 
BROWNAL &200&3&11&3&11&3&11&3&11&3&11&3&11\\ 
BROWNBS &2&31&34&31&33&31&33&32&34&33&35&32&34\\ 
BROWNDEN &4&29&45&24&39&26&43&18&34&23&40&18&34\\ 
BRYBND &5000&314&447&126&194&56&91&58&82&33&56&58&82\\ 
CHAINWOO &4000&628&940&222&317&173&244&47&72&90&142&47&72\\ 
CUBE &2&32&40&36&45&36&47&29&37&28&35&29&37\\ 
DECONVU &63&175&201&132&162&203&294&270&342&226&307&270&342\\ 
DENSCHNA &2&8&9&8&9&8&9&8&9&8&9&8&9\\ 
DENSCHNB &2&8&9&8&9&8&9&8&9&8&9&8&9\\ 
DENSCHNC &2&18&22&18&22&14&19&14&19&17&23&14&19\\ 
DENSCHND &3&8&21&8&21&8&21&6&20&12&25&6&20\\ 
DENSCHNE &3&8&10&8&10&8&10&8&10&9&12&8&10\\ 
DENSCHNF &2&10&14&10&14&9&14&9&14&13&22&9&14\\ 
DIXMAANA &3000&9&10&9&10&9&10&9&10&8&10&9&10\\ 
DIXMAANB &3000&59&77&10&12&10&12&10&12&8&10&10&12\\ 
DIXMAANC &3000&8&10&8&10&8&10&8&10&10&13&8&10\\ 
DIXMAAND &3000&51&63&9&14&9&14&9&14&12&17&9&14\\ 
DIXMAANE &3000&41&42&41&42&41&42&41&42&41&42&41&42\\ 
DIXMAANF &3000&78&86&155&163&36&38&36&38&36&38&36&38\\ 
DIXMAANG &3000&18&20&18&20&18&20&18&20&18&20&18&20\\ 
DIXMAANH &3000&283&303&50&56&50&56&44&50&45&57&44&50\\ 
DIXMAANI &3000&81&82&81&82&81&82&81&82&81&82&81&82\\ 
DIXMAANJ &3000&65&66&65&66&35&37&35&37&35&37&35&37\\ 
DIXMAANK &3000&19&21&19&21&19&21&19&21&19&21&19&21\\ 
DIXMAANL &3000&307&349&111&125&60&67&62&68&32&36&62&68\\ 
DJTL &2&166&266&166&266&169&270&174&274&168&271&174&274\\ 
DQDRTIC &5000&14&19&14&20&14&20&28&44&14&20&28&44\\ 
DQRTIC &5000&94&144&42&71&29&57&11&29&22&48&11&29\\ 
EDENSCH &2000&280&398&113&163&48&72&17&29&31&47&17&29\\ 
EG2 &1000&4&8&4&8&4&8&4&8&4&8&4&8\\ 
EIGENALS &2550&778&988&860&1150&1027&1275&1271&1493&1097&1354&1271&1493\\ 
ENGVAL1 &5000&280&391&129&180&38&53&16&26&34&53&16&26\\ 
ENGVAL2 &3&25&32&25&32&28&36&27&35&26&32&27&35\\ 
ERRINROS &50&249&356&97&151&76&116&36&55&68&96&36&55\\ 
EXPFIT &2&17&20&17&20&14&19&16&23&18&23&16&23\\ 
EXTROSNB &1000&506&746&403&571&443&525&27&45&447&546&27&45\\ 
GENROSE &500&4045&5861&3987&5672&4535&5803&4944&5823&4569&5652&4944&5823\\ 
GROWTHLS &3&18&32&17&31&25&41&24&41&21&36&24&41\\ 
GULF &3&26&29&26&32&26&32&26&32&26&32&26&32\\ 
HAIRY &2&62&75&32&42&25&34&36&45&23&29&36&45\\ 
HATFLDD &3&35&37&35&37&40&45&38&43&38&44&38&43\\ 
HATFLDE &3&9&12&9&12&9&12&9&12&10&14&9&12\\ 
HEART6LS &6&485&532&1366&1462&1551&1687&1938&2127&1740&1914&1938&2127\\ 
HEART8LS &8&273&297&281&312&315&367&287&336&304&348&287&336\\ 
HELIX &3&20&25&27&34&20&27&24&31&28&34&24&31\\ 
HIELOW &3&1&2&1&2&1&2&1&2&1&2&1&2\\ 
HILBERTA &2&6&7&6&7&6&7&6&7&6&7&6&7\\ 
HILBERTB &10&12&13&12&13&10&12&10&12&10&12&10&12\\ 
HIMMELBB &2&1&9&1&9&1&9&1&9&1&9&1&9\\ 
HIMMELBF &4&20&27&18&26&21&29&22&29&18&26&22&29\\ 
HIMMELBG &2&10&12&10&13&10&13&10&13&10&12&10&13\\ 
HIMMELBH &2&7&8&7&8&7&8&7&8&7&8&7&8\\ 
HYDC20LS &99&271&406&485&739&495&722&522&660&466&632&522&660\\ 
JENSMP &2&26&39&18&30&18&30&32&54&23&36&32&54\\ 
KOWOSB &4&31&32&31&32&28&29&28&29&30&32&28&29\\ 
LIARWHD &5000&20&27&20&27&20&27&20&27&20&27&20&27\\ 
LOGHAIRY &2&171&197&185&220&228&281&3608&4361&901&1094&3608&4361\\ 
MANCINO &100&512&693&99&146&42&71&12&32&28&55&12&32\\ 
MARATOSB &2&4&15&4&15&4&15&4&15&4&15&4&15\\ 
MEXHAT &2&12&29&12&29&12&29&12&29&12&29&12&29\\ 
MEYER3 &3&12&32&12&32&12&32&12&32&12&32&12&32\\ 
MSQRTALS &1024&4755&5900&4471&5947&4426&5643&4796&5541&4984&6147&4796&5541\\ 
MSQRTBLS &1024&4682&5870&4327&5723&4135&5244&4645&5404&4706&5755&4645&5404\\ 
NONDQUAR &5000&28&36&27&37&18&27&15&24&21&31&15&24\\ 
OSBORNEA &5&22&31&25&34&23&34&25&35&25&34&25&35\\ 
OSBORNEB &11&66&72&63&72&60&65&61&69&58&67&61&69\\ 
PALMER1C &8&6&20&6&20&6&20&7&23&6&20&7&23\\ 
PALMER1D &7&14&26&13&26&13&27&15&32&11&25&15&32\\ 
PALMER2C &8&5&17&5&17&5&17&5&17&5&17&5&17\\ 
PALMER3C &8&7&18&7&18&7&18&7&18&7&18&7&18\\ 
PALMER4C &8&11&22&11&22&11&22&9&21&11&22&9&21\\ 
PALMER5C &6&22&27&22&29&19&24&17&31&21&28&17&31\\ 
PALMER6C &8&10&19&10&19&10&19&10&20&10&19&10&20\\ 
PALMER7C &8&12&22&12&22&12&22&10&21&12&22&10&21\\ 
PALMER8C &8&13&22&13&22&13&22&17&27&13&22&17&27\\ 
PARKCH &15&1&2&1&2&1&2&1&2&1&2&1&2\\ 
PENALTY1 &1000&13&28&13&28&13&28&13&28&13&28&13&28\\ 
PENALTY2 &200&406&589&149&244&148&218&148&194&143&190&148&194\\ 
PENALTY3 &200&462&673&189&292&79&124&62&91&52&91&62&91\\ 
QUARTC &5000&94&144&42&71&29&57&11&29&22&48&11&29\\ 
ROSENBR &2&37&42&40&46&31&39&30&38&34&44&30&38\\ 
S308 &2&14&17&14&17&15&19&11&14&7&10&11&14\\ 
SENSORS &100&403&522&177&250&82&121&19&31&39&58&19&31\\ 
SINEVAL &2&70&80&78&88&76&87&83&102&84&99&83&102\\ 
SINQUAD &5000&23&31&21&30&20&29&28&39&21&30&28&39\\ 
SISSER &2&13&14&13&14&13&14&13&14&13&14&13&14\\ 
SNAIL &2&117&135&112&131&104&123&99&119&102&119&99&119\\ 
STRATEC &10&104&128&105&130&119&149&147&172&139&173&147&172\\ 
TOINTGOR &50&97&125&93&129&101&141&115&154&116&162&115&154\\ 
TOINTPSP &50&112&138&77&100&73&106&80&114&60&94&80&114\\ 
TOINTQOR &50&81&98&81&107&63&91&32&52&36&58&32&52\\ 
VARDIM &200&12&37&12&37&12&37&12&37&12&37&12&37\\ 
VAREIGVL &50&84&96&76&89&66&88&25&40&44&65&25&40\\ 
VIBRBEAM &8&140&188&84&120&59&87&67&102&76&105&67&102\\ 
WATSON &12&36&42&35&42&37&45&31&40&38&44&31&40\\ 
YFITU &3&55&64&61&73&52&63&66&80&64&77&66&80\\ 
ZANGWIL2 &2&3&4&3&4&3&4&3&4&3&4&3&4\\ 
 \hline
\end{longtable} 
\end{landscape} 
%%%%%%%%%%%%%%%
 
% ################################################
% ################################################ 

\end{document}